\documentclass[a4paper,11pt]{article}
\usepackage[utf8]{inputenc}
\usepackage{graphicx}
\usepackage[top=2cm, bottom=2cm, left=2.5cm, right=2.5cm]{geometry}
\usepackage{amsmath,amsthm,amsfonts,amssymb,amscd, dsfont}
\usepackage{lastpage}
\usepackage{enumerate}
\usepackage{fancyhdr}
\usepackage{mathrsfs}
\usepackage{xcolor}

\usepackage{listings}
\usepackage{subcaption}
\usepackage{hyperref}

\usepackage{imakeidx}
\usepackage{multirow}
\usepackage{float}
\usepackage{MnSymbol}
\usepackage{wasysym}
\usepackage{tikz}
\usepackage{tikz-cd}
\usepackage[T1]{fontenc}

\usepackage{authblk}
\usepackage{dirtytalk}
\usepackage[normalem]{ulem}

\DeclareMathOperator{\Homeo}{Homeo}
\DeclareMathOperator{\Aff}{Aff}
\DeclareMathOperator{\SL}{SL}
\DeclareMathOperator{\Dil}{Dil}

\DeclareMathOperator{\GL}{GL}
\DeclareMathOperator{\e}{end}

\newtheorem{theorem}{Theorem}[section]

\newtheorem{lemma}[theorem]{Lemma}
\newtheorem{definition}[theorem]{Definition}

\title{Horizon saddle connections and Veech groups of infinite-type dilation surfaces}

\author{Oscar Rutilio Molina Medrano}

\date{}

\begin{document}

\maketitle

\begin{abstract}
     This paper  studies Veech groups of dilation surfaces whose fundamental group is not finitely generated.  We prove that if $ S$ is a surface with  self-similar end space and every end is accumulated by genus, then every countable subgroup of $ \SL(2,\mathbb{R}) $ can be realized as the Veech group of a dilation surface homeomorphic to $ S$. Additionally, we can construct this dilation surface such that it realizes horizon saddle connections in a prescribed set of directions. This contrasts with the case of closed dilation surfaces, where horizon saddle connections impose restrictions on the algebraic structure of the Veech group.

\end{abstract}

\section{Introduction}\label{sec1}

    By a topological surface we mean a connected, orientable, and boundaryless $2$-dimensional real manifold.  A topological surface is of \textit{finite type} if its fundamental group is finitely generated. Otherwise, the surface is of \textit{infinite-type}. Some examples of infinite-type surfaces can be seen in Figure \ref{Figure_Infinite_type_surfaces}. 

    A \textit{dilation surface} is a topological surface endowed with a complex-valued atlas whose transition maps, except in the neighborhoods of a discrete set of singular points, are dilations composed with translations. In the particular case where all  transition maps are translations,  the dilation surface is a \textit{translation surface}

   Given a dilation surface $M$, the group consisting of the derivatives of orientation preserving transformations preserving the dilation structure on $M$ is a subgroup of $\GL^+(2,\mathbb{R})$. This group of matrices is known as the \textit{Veech group} of $M$. 
   
   The Veech group of finite-type dilation surfaces have been widely study and there are many interesting classical results about it. For example the Veech group of a closed translation surface isa  Fuchsian group that is never cocompact, the lector can read this and more interesting results about finite-type translation surfaces and their Veech groups in \cite{HS06} and \cite{Zor06}.
    
    More recently, there have been some results concerning the Veech groups of infinite-type translation surfaces. In \cite{AMRSVW23} Artigiani, Randecker, Sadanand, Valdez, and Weitze-Schmith\"usen showed that any countable subgroup of $\GL^+(2,\mathbb{R})$ can be realized as the Veech group of an infinite-type translation surface with self-similar space of ends, such that every end is accumulated by genus.

    On the other hand, dilation surfaces have historically been less studied, although interest in them has increased in recent years. The first authors to study the Veech groups of dilation surfaces were Eduard Duryev, Charles Fougeron, and Selim Ghazouani. In \cite{DFG19}, these authors proved that the Veech group of a closed dilation surface of genus $g\ge 2$ is either discrete or the subgroup of $\SL(2,\mathbb{R})$ consisting of upper triangular matrices. Moreover, they provided a complete description of which dilation surfaces have a non-discrete Veech group. Later, in \cite{Tahar21}, Guillaume Tahar introduced the notion of \textit{horizon saddle connections}. A horizon saddle connection is a saddle connection $\gamma$ for which there exists an integer $k$ such that for every leaf $\alpha$ of a directional foliation, $|\gamma\cap \alpha|\leq k$. Tahar proved that the existence of these rigidifies the algebraic structure of the Veech groups of closed dilation surfaces. In particular, if a closed dilation surface has horizon saddle connections in three different directions, the Veech group must be finite.

     The results we present constitute the first formal study of infinite-type dilation surfaces and their Veech groups. Adapting ideas  the work of Artigiani, Randecker, Sadanand, Valdez, and Weitze-Schmithüsen in \cite{AMRSVW23}, which focused on translation surfaces, we show that the existence of horizon saddle connections does not impose algebraic restrictions on the Veech group of infinite-type dilation surfaces. In particular, we prove that any countable subgroup of $\SL(2,\mathbb{R})$ can be realized as the Veech group of an infinite-type dilation surface. Moreover, these surfaces can be constructed so as to admit horizon saddle connections in a prescribed set of directions and to be homeomorphic to surfaces with particular topological properties. 

    \begin{theorem}\label{Main Theorem 1}
    Let $S$ be a topological surface with infinite genus and self-similar end-space. Let $G<\SL(2, \mathbb{R})$ be an infinite-countable subgroup and $D\subset \mathbb{S}^1$ a countable set of directions and $\{k_d\}_{d\in D}$ a sequence of positive integers. Then, there exists an uncountable family of non-isomorphic dilation surfaces $\{M_r\}_{r\in(0,1)}$ such that for each $r\in(0,1)$, $M_r$ is homeomorphic to $S$ and the following statements hold:

    \begin{enumerate}
        \item $\Gamma(M_r)=G$.
        \item  $\Aff(M_r) \cong G$.
        \item If $D\neq \emptyset$, then for every $d\in D$ there exists  a $k_d$-horizon saddle connection $\gamma_d\subset M_r$ with direction $d$.
    \end{enumerate}
\end{theorem}

Theorem~\ref{Main Theorem 1} establishes a realization result for infinite countable subgroups of $\SL(2,\mathbb{R})$ as the Veech group of a dilation surface with infinite-genus and self-similar end space. On the other hand, Theorem~\ref{Main Theorem 2} shows that, by changing the topology of the surface, we can also realize finite subgroups of $\SL(2,\mathbb{R})$ as Veech groups. In both cases, the groups can be realized as the Veech groups of an infinite-type dilation surface with horizon saddle connections appearing in infinitely many directions, contrasting with the results of Tahar~\cite{Tahar21} for closed surfaces.

\begin{theorem}\label{Main Theorem 2}
    Let $S$ be a topological surface with  $E(S)=E^g(S)$, let $G< \SL(2,\mathbb{R})$ be a finite group and $D\subset \mathbb{S}^1$ a countable set of directions and $\{k_d\}_{d\in D}$ a sequence of positive integer. Then, there exists an uncountable family of non-isomorphic dilation surfaces $\{M_r\}_{r\in(0,1)}$ such that for each $r\in(0,1)$, $M_r$ is homeomorphic to $S$ and the following statements hold

    \begin{enumerate}
        \item $\Gamma(M_r)=G$.
        \item  $\Aff(M_r) \cong G$.
        \item If $D\neq \emptyset$, then for every $d\in D$ there exists  a $k_d$-horizon saddle connection $\gamma_d\subset M_r$ with direction $d$.
    \end{enumerate}
\end{theorem}

Finally, while Theorems~\ref{Main Theorem 1} and~\ref{Main Theorem 2} 
show that any countable subgroup of $\SL(2,\mathbb{R})$ can appear as the Veech group of an infinite-type dilation surface with 
particular topological properties, the next result shows that any infinite-type topological surface admits a dilation structure whose Veech group is uncountable.

    \begin{theorem}\label{Main Theorem 3}
    Let $S$ be an infinite-type topological surface and $$G:= \left\{\begin{pmatrix}
                a & b \\
                0 & a^{-1} 
                \end{pmatrix}: a\in \mathbb{R}^+, b\in \mathbb{R}\right\}.$$ Then  there exists an uncountable family of non-isomorphic dilation surfaces $\{M_r\}_{r\in(0,1)}$  such that for all $r\in(0,1)$, $M_r$ is homeomorphic to $S$ and $\Gamma(M_r)=G$ .

\end{theorem}

\section*{Acknowledgements}
This work was supported by SECIHTI, UNAM-PAPIIT IN106925, and IRL 2001 CNRS-UNAM (Solomon Lefschetz). The author thanks the Institut de Math\'ematiques de Toulouse for hospitality and accommodation during a research visit, and Ferr\'an Valdez for invaluable discussions and advice.

During the preparation of this work, the author used Gemini (Google) 
for typesetting assistance and English language editing (grammar, punctuation,  and readability). The author reviewed and verified all content and takes full responsibility for the final mathematical results and formulations.

\section{Topological preliminaries.}

In this section we discuss some topological preliminaries that will be necessary for the language used in the statement of the main theorems and their proofs: The space of ends of a surface, the Mann-Rafi order, self-similarity and radial symmetry.

\subsection{The space of ends.}

The space of ends of a surface $S$, denoted $E(S)$, encodes the different ways in which one can approach infinity within the surface. There is a distinguished class of ends consisting of those that can be approached through arbitrary large genus. Such an end is called   \textit{accumulated by genus}, denote by $E^g(S)$ the subspace of $E(S)$ that consists of all the ends that are accumulated by genus.

$E^g(S)$ is a closed subset of $E(S)$ and plays an important role in the topological classification of surfaces. For this reason, we will think of the space of ends of $S$ as a pair of nested spaces $E^g(S)\subset E(S)$. The following theorem is a combination of the results of Ker\'ekj\'art\'o and Richards, see \cite{Richards}.

 \begin{theorem}\label{Theorem classification of surfaces}
 Let $S_1$ and $S_2$ be orientable  surfaces with empty boundary. Then $S_1$ and $S_2$ are homeomorphic if and only if both have the same genus and  $E^g(S_1)\subset E(S_1)$ is homeomorphic as  nested spaces \footnote{That is, there exists a homeomorphism $f: E(S_1)\to E(S_2)$ such that its restriction $f|_{E^g(S_1)}: E^g(S_1)\to E^g(S_2)$ is a homeomorphism.} to $E^g(S_2)\subset E(S_2)$. Moreover, given $g\in \mathbb{N}_0\cup\{\infty\}$ and $A\subset B$ closed subsets of the Cantor set, there exists a surface $S$ with genus $g$ and space of ends homeomorphic to $A\subset B$.
\end{theorem}

In contrast to orientable finite-type surfaces, whose topology is determined by the genus and the number of ends, infinite-type surfaces require specifying the topology of $E^g(S)\subset E(S)$. This gives rise to pretty names for special surfaces. See Figure \ref{Figure_Infinite_type_surfaces}.

\begin{figure}
  
    \begin{subfigure}[t]{0.32\textwidth}
        \centering
        \includegraphics[height=2cm]{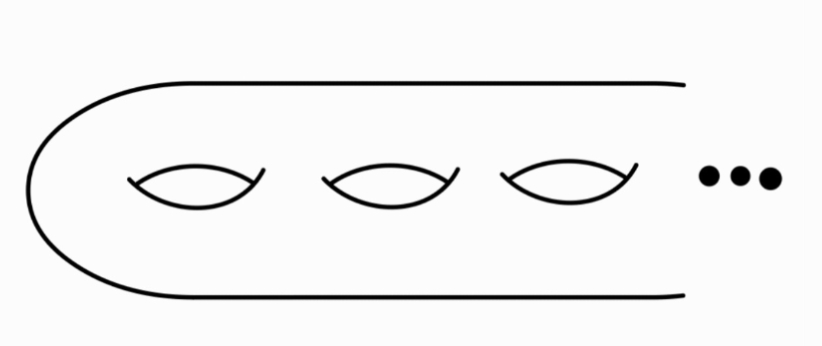}
        \caption{The Loch Ness monster, is the surface with infinite genus and exactly one end.}
        \label{fig:sub1}
    \end{subfigure}
    \hfill
    \begin{subfigure}[t]{0.32\textwidth}
        \centering
        \includegraphics[height=4cm]{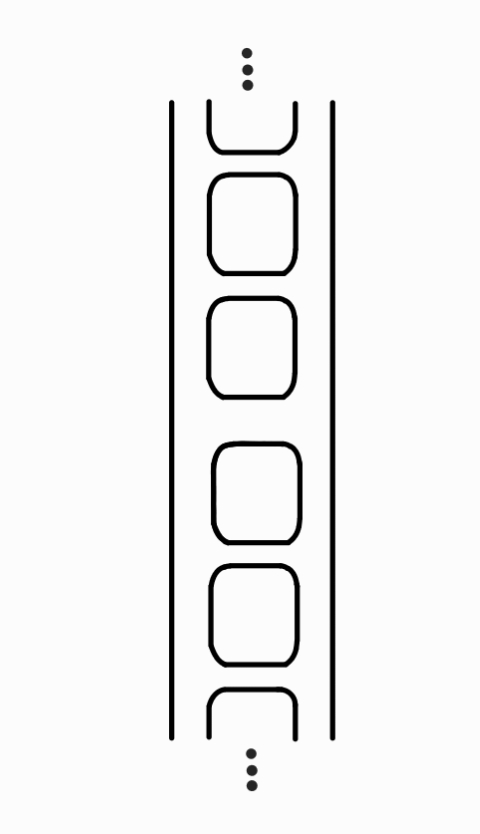}
        \caption{The Jacob´s ladder, is the surface with infinite genus and exactly two ends, both accumulated by genus.}
        \label{fig:sub2}
    \end{subfigure}
    \hfill
    \begin{subfigure}[t]{0.32\textwidth}
        \centering
        \includegraphics[height=4cm]{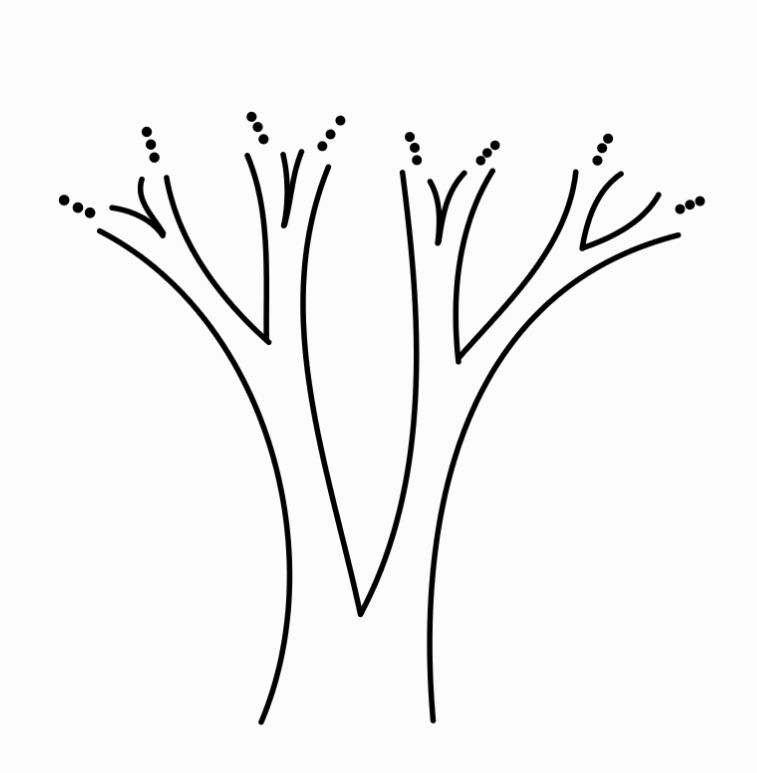}
        \caption{The Cantor Tree, is the surface with genus zero and space of ends homeomorphic to the Cantor space.}
        \label{fig:sub3}
    \end{subfigure}
  \caption{Some examples of infinite-type surfaces: (a) the Loch Ness monster, is the surface with infinite genus and exactly one end, (b) the Jacob´s ladder, is the surface with infinite genus and exactly two ends, both accumulated by genus (c) the Cantor Tree, is the surface with genus zero and space of ends homeomorphic to the Cantor space.}
  \label{Figure_Infinite_type_surfaces}
\end{figure}

\subsection{The Mann-Rafi order.}

There is a way to endow the space of ends (modulo an equivalence relation) of an infinite-type surface with a partial order based on the local complexity of the ends. This order was introduced by Mann and  Rafi in \cite{MR23}.\\

Let $S$ be a surface, a homeomorphism  $f\in\Homeo(S)$ induces a homeomorphism $f^*\in \Homeo(E(S))$. Given $x,y\in E(S)$, we say that $x\preccurlyeq y$ if for every neighborhood $V$ of $y$, there exist a neighborhood $U$ of $x$ and $f\in \Homeo(S)$ such that $f^*(U)\subset V$. We write $x\sim y$ if $x\preccurlyeq y$ and $y\preccurlyeq x$. Let us define $$
\mathcal{E}(x):=\{y\in E(S): x\sim y\}, \qquad
\mathcal{E}(S):=\{\mathcal{E}(x): x\in E(S)\}.
$$

The relation $\prec $ on $E(S)$, defined by $x\preccurlyeq y$ and $x\not\sim y$, induces a partial order on $\mathcal{E}(S)$, called the \textbf{Mann-Rafi order}. As Mann and Rafi shown in 
{\cite[Proposition~4.7]{MR23}}, this order has maximal elements.

\subsection{Self-similarity and radial symmetry}

We now describe the class of infinite-type surfaces to which the main theorems apply. These are surfaces whose space of ends exhibits a fractal-like behavior.  Formally, the space of ends satisfies two equivalent notions: self-similarity and radial symmetry.

\begin{definition}
    Let $S$ be a surface. We say that the space of ends of $S$ is \textbf{self-similar} if for every decomposition   $E(S)=E_1\sqcup E_2$ into disjoint clopen sets, there exist $i\in \{1,2\}$ and $A\subset E_i$ such that $A\cap E^g(S)\subset A$ is homeomorphic as nested spaces to $E^g(S)\subset E(S)$.
\end{definition}

The Loch Ness Monster and the Cantor Tree have a self-similar space of ends, while Jacob´s ladder does not (see Figure \ref{Figure_Infinite_type_surfaces}).

Aougab, Patel and Vlamis introduced the notion of \textit{radial symmetry} in \cite{APV21} and proved its equivalence with self-similarity.

\begin{definition}
    Let $S$ be a topological surface. $E(S)$ has \textbf{radial symmetry} if either $E(S)$ consists of only one point, or there exist
    $x_\infty\in E(S)$ and a countable collection $\{E_n\}_{n\in\mathbb{N}}$  of pairwise homeomorphic non-compact subsets of $E(S)$ such that $$
E(S)\setminus\{x_\infty\}=\bigsqcup_{n\in\mathbb{N}} E_n
\quad\text{and}\quad
\overline{E_m}\cap\overline{E_n}=\{x_\infty\}\ \text{for } m\neq  n.
$$

    In this case we say that $x_\infty$ is a \textbf{star point} of $E(S)$ and we call the family $\{E_n\}_{n\in\mathbb{N}}$ a \textbf{star decomposition} of $E(S)$ with respect to the star point $x_{\infty}$.
\end{definition}

\begin{theorem}\label{Radial_Symmetry_Self_Similar_Equivalence}{\cite[Theorem~5.2 and Proposition~5.15]{APV21}}
   If  $S$ is an infinite-type surface. Then the following statements hold:

   \begin{enumerate}
       \item $E(S)$ is self-similar if and only if it has radial symmetry.
       \item If $E(S)$ has radial symmetry, then $x\in E(S)$ is a star point if and only if $\mathcal{E}(x)$ is maximal with respect to the Mann-Rafi order.
   \end{enumerate}
\end{theorem}

\section{Geometrical preliminaries.}

This section aims to provide basic definitions and geometric properties of translation and dilation surfaces, as well as what is known about the Veech groups of dilation surfaces.

\subsection{Dilation surfaces and conic singularities.}

We begin by introducing dilation surfaces and the kind of singularities that we are considering in this . We first define dilation atlases on topological surfaces and then, we describe the local models arising at singular points.

\begin{figure}
  \centering
    \includegraphics[width=0.7\linewidth]{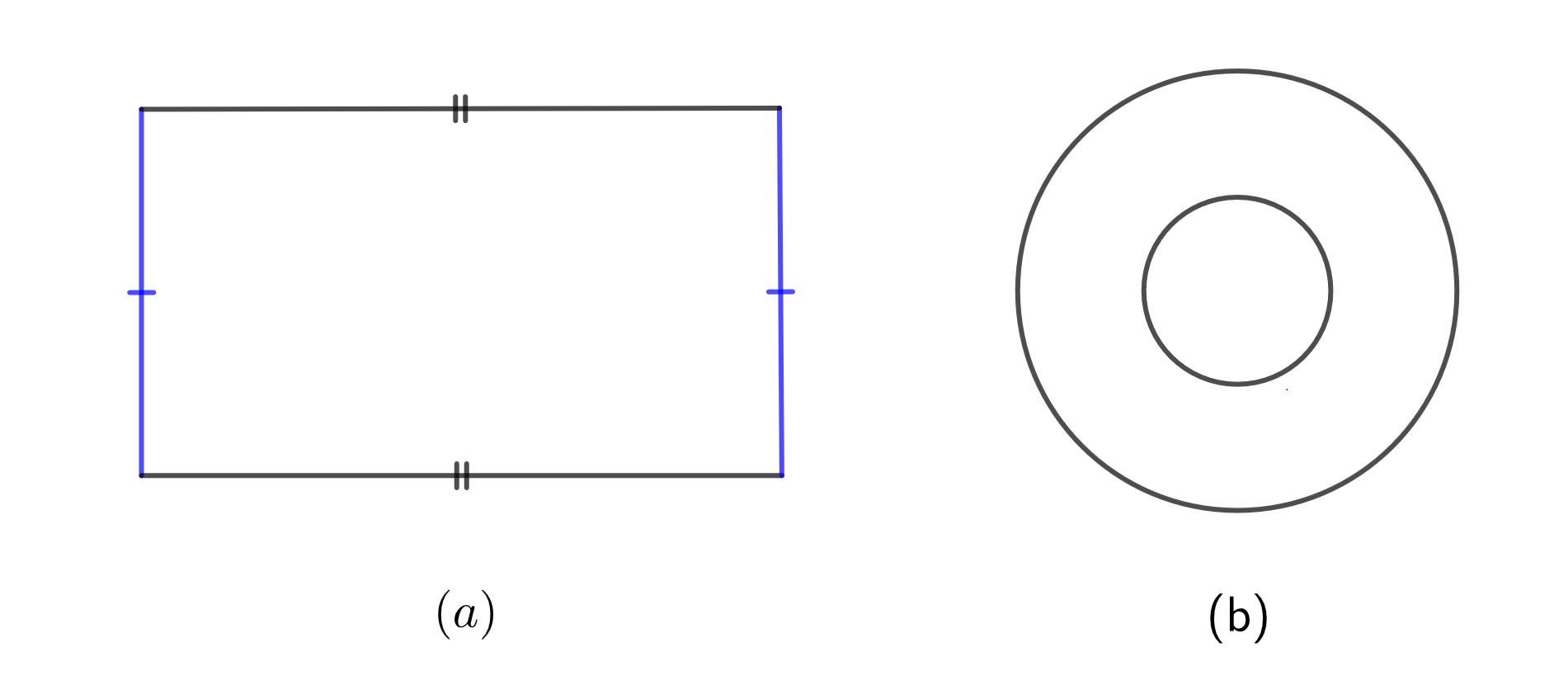}
  \caption{Some quotients that induce dilation atlases on the torus: (a) induces a pure translation atlas, while (b) induces a dilation atlas that is not a translation atlas.}
  \label{Figure_Dilation_Surfaces_Examples}
\end{figure}

\begin{definition}
    Let $S$ be a topological surface. A \textbf{dilation atlas} on $S$ is a complex atlas on $S$ whose transition maps are of the form $z\mapsto r z+v$ with $r\in\mathbb{R}^+$ and $v\in\mathbb{C}$. 
\end{definition}

In Figure \ref{Figure_Dilation_Surfaces_Examples} (a) the identifications in the rectangle by translations endows the Torus with a dilation atlas where all the transition maps are translations. Another dilation atlas in the Torus is induced by the quotient $\mathbb{C}^*/(z\sim r z)$ with $r\in\mathbb{R}^+\setminus\{1\}$ a fixed number. Figure \ref{Figure_Dilation_Surfaces_Examples} (b) shows a fundamental region for this quotient. \\                                                                                        

Note that not every topological surface admits a dilation atlas; for example, the sphere $\mathbb{S}^2$ does not. Moreover, the only closed orientable surface admitting a dilation atlas is the torus. For this reason, we will consider a dilation atlas on the complement of a discrete set of singularities. These singularities can be modeled as singularities on dilation cones.

Let $k\in \mathbb{N}$ and $\lambda\in \mathbb{R}^+$. 
A \textit{dilation cone} of angle $2k\pi$ and holonomy $\lambda$ is obtained by taking a $k$--fold cyclic branched cover of the unit disk $\mathbb{D}$, fully ramified at the origin, and then cutting along a straight half-line issuing from the cone point and gluing the two sides by a dilation of factor $\lambda$. The resulting surface has a single cone point of total angle $2k\pi$ and holonomy $\lambda$. In the special case $\lambda=1$, the construction yields a \textit{flat cone} of angle $2k\pi$.

\begin{definition}
    Let $S$ be a topological surface and $p\in S$. Given a dilation atlas $\mathcal{A}$ on $S\setminus\{p\}$, we say that $p$ is a \textbf{conic singularity} of angle $2k\pi$ and holonomy $\lambda$ with respect to the atlas $\mathcal{A}$ if the following condition holds:
    
        \begin{enumerate}
            \item $k\geq 2$ or, if $k=1$, then $\lambda\neq 1$.

    \item There exists an open neighborhood $U_p$ of $p$ and a homeomorphism
    $f$ from $U_p$ to  a dilation cone of angle $2k\pi$ and holonomy $\lambda$ such that $f(p)$ is the cone point and the restriction
    $f|_{U_p\setminus\{p\}}$ is, in local coordinates, of the form
    $z\mapsto rz+v$ with $r\in\mathbb{R}^+$ and $v\in\mathbb{C}$.
        \end{enumerate}
\end{definition}

We say that a conic singularity has \textit{trivial holonomy} if $\lambda=1$.
Throughout the paper, whenever we refer to a conic singularity of angle $2k\pi$ without specifying its holonomy, we will assume that its holonomy is trivial.

\begin{definition}
    A \textbf{dilation surface} is a triplet $M=(S, \Sigma, \mathcal{A})$ where: 

    \begin{enumerate}
        \item $S$ is a topological surface and $\Sigma\subset S$ is a discrete set.
        \item $\mathcal{A}$ is a maximal dilation atlas on $S\setminus \Sigma$.
        \item Every point $p\in \Sigma$ is a conic singularity with respect to the atlas $\mathcal{A}$.
    \end{enumerate}
\end{definition}

Translation surfaces are examples of dilation surfaces for which all singularities (if any) have trivial holonomy. An example of a dilation surface without singularities is the  Hopf torus $T_r:=\mathbb{C}^*/(z\sim r z)$, where $r\in\mathbb{R}^+\setminus \{1\}$. Below we describe a surgery process for constructing dilation surfaces from existing ones. This process, known as \textit{slit construction}, will play a key role in the proofs of the main theorems.

\begin{figure}[t]
    \centering
    \includegraphics[width=0.7\linewidth]{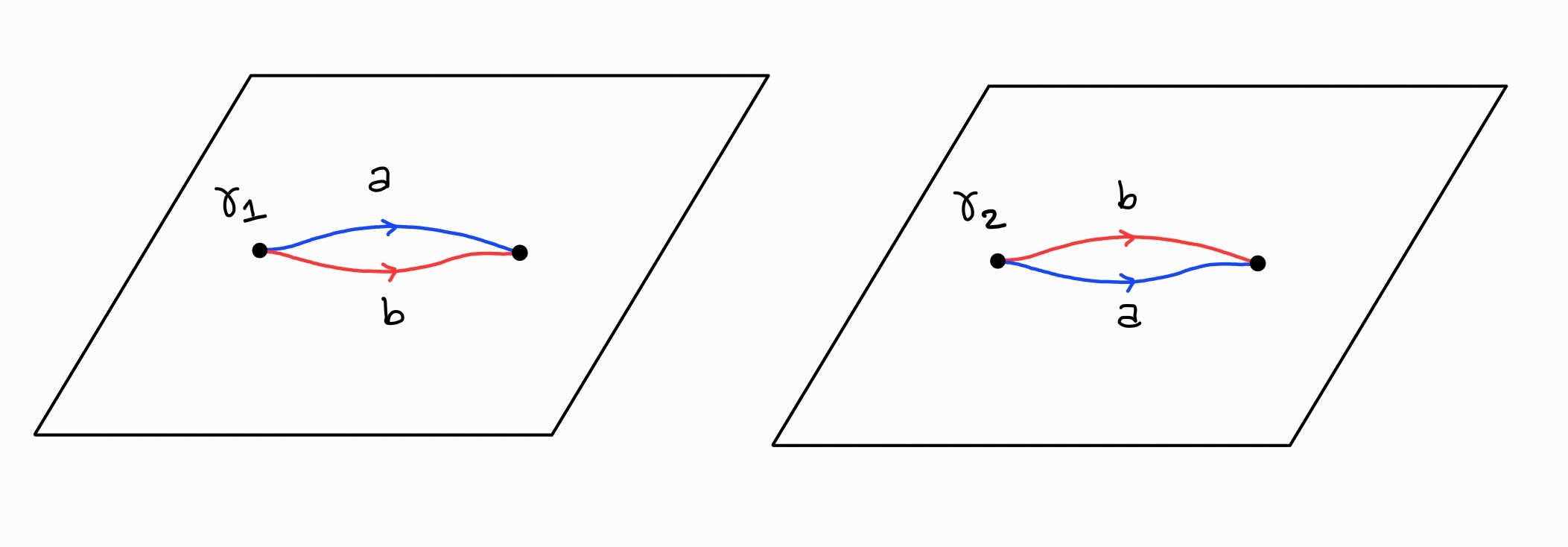}
    \caption{Gluing throughout two slits}
    \label{Figure_Slit_Construction}
\end{figure}

Let $M_1$ and $M_2$ be dilation surfaces with singular sets $\Sigma_1$ and $\Sigma_2$ respectively. Let $\gamma_i\subset M_i\setminus \Sigma_i, i=1,2$, be parallel compact straight segments of finite length. We cut $M_i$ along $\gamma_i$ and glue the left side of $\gamma_1$ with the right  side of $\gamma_2$ and vice versa. We obtain a dilation surface $M$ with singularities $\Sigma_1\cup\Sigma_2\cup \{\sigma_1,\sigma_2\}$where $\sigma_1$ and $\sigma_2$ are conic singularities of angle $4\pi$  corresponding  to the end points of the glued segments. Figure \ref{Figure_Slit_Construction} illustrates this process.\\

It is natural to study the class of maps that respect the dilation structure. This leads to the notion of a dilation and affine homeomorphisms between dilation surfaces, which play a central role in the study of their symmetries and, in particular, in the definition of the Veech group.

\begin{definition}\label{Definition of affine homeo}
    Let $M_i=(S_i, \Sigma_i, \mathcal{A}_i), i\in \{1,2\}$ be dilation surfaces. An \textbf{affine homeomorphism} from $M_1$ to $M_2$ is a homeomorphism $f:S_1\to S_2$ such that $f(\Sigma_1)=\Sigma_2$ and whose restriction $f|_{S_1\setminus \Sigma_1}$ is, in local coordinates given by the form $z\mapsto r A_fz+v$ with $A_f\in \SL(2,\mathbb
    R), r\in \mathbb{R}^+$ and $v\in\mathbb{C}$.

    If $A_f$ is the identity matrix, then we say that $f$ is a \textbf{dilation}. We say that two dilation surfaces $M_1$ and $M_2$ are \textbf{isomorphic} if there exists a dilation from $M_1$ to $M_2$.

    Let $M$ be a dilation surface. Denote by $\Aff(M)$ the group that consists of all the affine homeomorphisms from $M$ to itself. And denote by $\Dil(M)$ the group that consists of all dilations from $M$ to itself. 
\end{definition}

Because of the structural group in a dilation surface, the matrix $A_f$ in Definition \ref{Definition of affine homeo} does not depend on the choice of local charts (changing coordinates can only change the factor $r$).\\

Note that for compact dilation surfaces, the slit construction produces a dilation surface with two extra singularities, in particular the resulting dilation surface is not isomorphic to the initial ones.

\subsection{The linear holonomy.}

Each dilation surface is endowed with a representation of the fundamental group of the underlying surface punctured at the singularities into $\mathbb{R}^+$; this representation is called the \textit{linear holonomy} and measures how the geometric structure is rescaled when traveling along a loop.

Let $M=(S,\Sigma,\mathcal{A})$ be a dilation surface and let $\gamma$ be a simple closed loop in $M\setminus\Sigma$. Cover $\gamma$ with finitely many charts $\{(U_i,\varphi_i)\}_{i=0}^k\subset\mathcal{A}$ such that $U_i\cap U_{i+1}\neq \emptyset$.  Assume that each transition map is of the form $z\mapsto r_i z+v_i$, with $r_i\in\mathbb{R}^+$ and $v_i\in\mathbb{C}$. The \textbf{linear holonomy} of $\gamma$ is defined by
$$
\overline{\rho}(\gamma):=\prod_{i=0}^{k} r_i.
$$

The linear holonomy is independent of the chosen charts and is invariant under free homotopy of loops, (see {\cite[Section~3.4]{Thu97}}) and thus defines a representation of $\pi_1(M\setminus\Sigma)$ into $\mathbb{R}^+$.

\begin{definition}
    Let $M=(S, \Sigma, \mathcal{A})$ be a dilation surface. The \textbf{linear holonomy} of $M$ is the representation \begin{align*}
    \rho_M: \pi_1(S\setminus \Sigma)&\to\mathbb{R}^+\\
    [\gamma]&\mapsto \overline{\rho}(\gamma).
\end{align*}
\end{definition}

Note that if $M$ is a dilation cone of angle $2k\pi$ and holonomy $\lambda$ and $\gamma$ is the generator of the fundamental group of this cone punctured on the cone point, then $\rho_M(\gamma)=\lambda$. Particularly, we can calculate the holonomy of a singularity in a dilation surface through  the linear holonomy.\\

The following lemma follows from the maximality of the dilation atlas and will be used in the proof of the main theorems.

\begin{lemma}\label{Equivalence and linear holonomy}
    Let $M_1=(S_1, \Sigma_1, \mathcal{A}_1)$ and $M_2= (S_2, \Sigma_2, \mathcal{A}_2)$ be dilation surfaces with linear holonomy $\rho_{M_1}$ and $\rho_{M_2}$, respectively. If $f$ is an affine homeomorphism from $M_1$ to $M_2$, then for each $[\gamma]\in \pi_1(M_1\setminus \Sigma_1)$ we have
    $$ \rho_{M_1}([\gamma])=\rho_{M_2}([f(\gamma)]).$$ In particular, the holonomy of a conic singularity is invariant under affine homeomorphisms.
\end{lemma}

\begin{proof}
    Let $\{(U_i, \varphi_i)\}_{i=0}^k\subset \mathcal{A}_1$ be a collection of charts that covers $\gamma$ and such that $U_i\cap U_{i+1}\neq  \emptyset$. Assume that the transition map between $(U_i, \varphi_i)$ and $(U_{i+1}, \varphi_{i+1})$ is of the form $z\mapsto r_iz+v_i$ for some $r_i\in \mathbb{R}^+$ and $v_i\in\mathbb{C}$. 
    
    Let $A_f\in\SL(2, \mathbb{R})$ denote the matrix in Definition \ref{Definition of affine homeo} associated to $f$, then the respective matrix associated to $f^{-1}$ is $A_f^{-1}$. It follows from the maximality of the dilation atlases that if we set $V_i:=f(U_i)$, $ \phi_i:=A_f \circ \varphi_i\circ f^{-1}$, then $(V_i, \phi_i)\in\mathcal{A}_2$. Moreover, the transition map between $(V_i, \phi_i)$ and $(V_{i+1}, \phi_{i+1})$ is of the form $z\mapsto r_i z+ v_i^\prime$ with $v_i^\prime:=A_f\cdot v_i$. Thus, it follows that $$\rho_{M_1}([\gamma])=\prod_{i=0}^{k} r_i=\rho_{M_2}(f([\gamma])).$$
\end{proof}
Lemma~\ref{Equivalence and linear holonomy} shows that linear holonomy is invariant under affine homeomorphisms; however, it does not completely determine the geometric structure of a dilation surface. For instance, every translation surface has trivial linear holonomy.\\

\subsection{Directional foliations.}

Because of the structure group of the dilation atlas, it makes sense to speak about angles and directions on a dilation surface. Thus, for any dilation surface $M$ and any direction $\theta\in \mathbb{S}^1$, there exists a singular foliation $\mathcal{F}_\theta$, called the \textit{directional foliation}, whose leaves are curves of constant direction $\theta$ and whose singularities are precisely the conic singularities of $M$.

\begin{definition}
Let $M$ be a dilation surface with singular set $\Sigma$.
A \textbf{saddle connection} on $M$ is a leaf of a directional
foliation whose endpoints are singularities and whose interior
does not contain singularities.
\end{definition}

By its definition in local coordinates, dilations and affine homeomorphisms preserve some geometric aspects of dilation surfaces: 

\begin{itemize}
    \item Each dilation preserves angles and directions.
    \item Each dilation sends leaves of a foliation $\mathcal{F}_\theta$ in leaves of $\mathcal{F}_\theta$.
    \item Each affine homeomorphism sends (if there exist) conic singularities of angle $2k\pi$ in conic singularities of angle $2k\pi$ and preserves the holonomy.
    \item Each affine homeomorphism sends leaves of a directional foliation $\mathcal{F}_{\theta}$ in leaves of a directional foliation $\mathcal{F}_{\theta^\prime}$. In particular, it sends saddle connections to saddle connections.
\end{itemize}

\subsection{Horizon saddle connections}

 \begin{figure}
     \centering
     \includegraphics[width=0.6\linewidth]{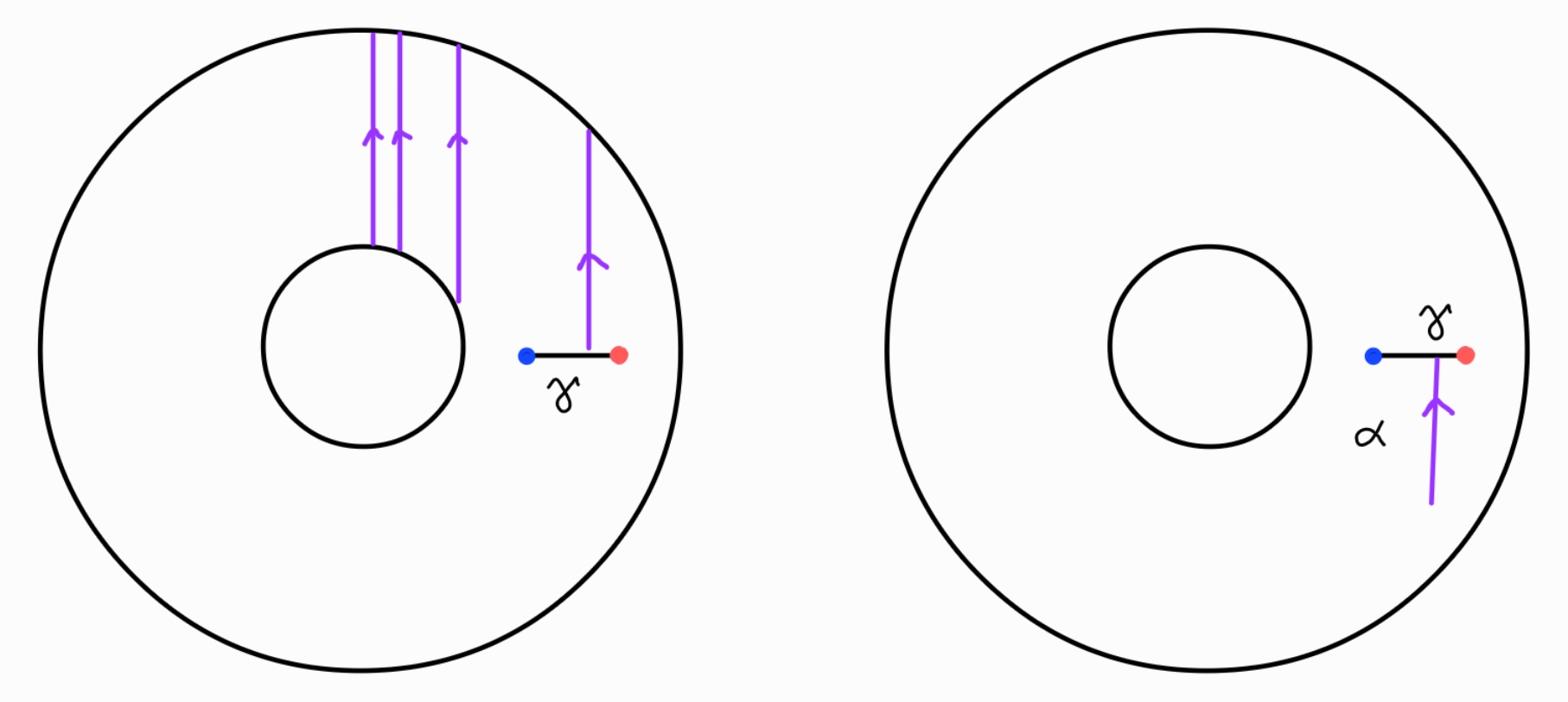}
     \caption{$\gamma$ is an example of a 1-horizon saddle connection; each leaf of a directional foliation intersects $\gamma$ at most once.}
     \label{Figure_Example_Saddle_connection}
 \end{figure}

There is a phenomenon that occurs in some compact dilation surfaces but not in compact translation surfaces:  the existence of a special type of saddle connection that controls the maximum number of intersections with leaves of directional foliations.  This notion was introduced by Tahar in \cite{Tahar21}.

\begin{definition}
    Let $M$ be a dilation surface. A $k$-\textbf{horizon saddle connection} is a saddle connection $\gamma$ on $M$ such that for any $\theta\in[0,2\pi)$, if $\alpha $ is a leaf of $\mathcal{F}_\theta$ transversal to $\gamma$, then $|\gamma\cap \alpha  |\leq k$.
\end{definition}

We now describe a  construction of surfaces with a 1-horizon saddle connection.\\

 Given $r\in\mathbb{R}^+\setminus\{1\}$, let $T_\lambda $ denote the Hopf torus $\mathbb{C}^*/(z\sim r z)$ and let $\rho : \mathbb{C}^*\to T_r $ be the natural covering map. Given $\theta\in [0,2\pi]$, a \textbf{radial segment} in $T_r $ in direction $\theta$ is a connected component of $\rho(\mathbb{R}^*\cdot e^{i\theta})$.

 Now, consider two Hopf tori $T_{r_1}$ and $T_{r_2}$. Let  $\gamma_1\subset T_{r_1}$ and $\gamma_2\subset T_{r_2}$ be parallel straight segments contained in radial segments. Performing the slit construction gluing $\gamma_1\subset T_{r_1}$ and $r\subset T_{r_2}$ we obtain a dilation surface $M$. We claim that the resulting saddle connection $\gamma$ is a $1$-horizon saddle connection. Indeed, if $\alpha$ is a leaf of a directional foliation $\mathcal{F}_\theta$ transversal to $\gamma$ and $\gamma\cap\alpha\neq \emptyset$, then $\alpha$ accumulates within a Hopf torus onto a radial segment of direction $\theta$ (see Figure \ref{Figure_Example_Saddle_connection}).\\

The construction of $k$-horizon saddle connection with $k\geq 2$ is more complicated.  Lemma \ref{Construction of k horizon saddle connections} will give us a method to construct a  dilation surface with a $k$-horizon saddle connection. See Figure \ref{Figure 2-horizon} for an example of a 2-horizon saddle connection.

\begin{lemma}\label{Construction of k horizon saddle connections}
    
         Given a positive integer  $k$, there exists a dilation surface $M$ with the following properties:

    \begin{enumerate}
        \item $M$ has only one conic singularity and it is of angle  $2\pi (10k+15)$.
        \item For each $l\in \{1,\dotsb, k\}$ there exists $\gamma_l\subset M$ a $l$-horizon saddle connection.
        
    \end{enumerate}
\end{lemma}

\begin{proof}

\begin{figure}
     \centering
     \includegraphics[width=0.8\linewidth]{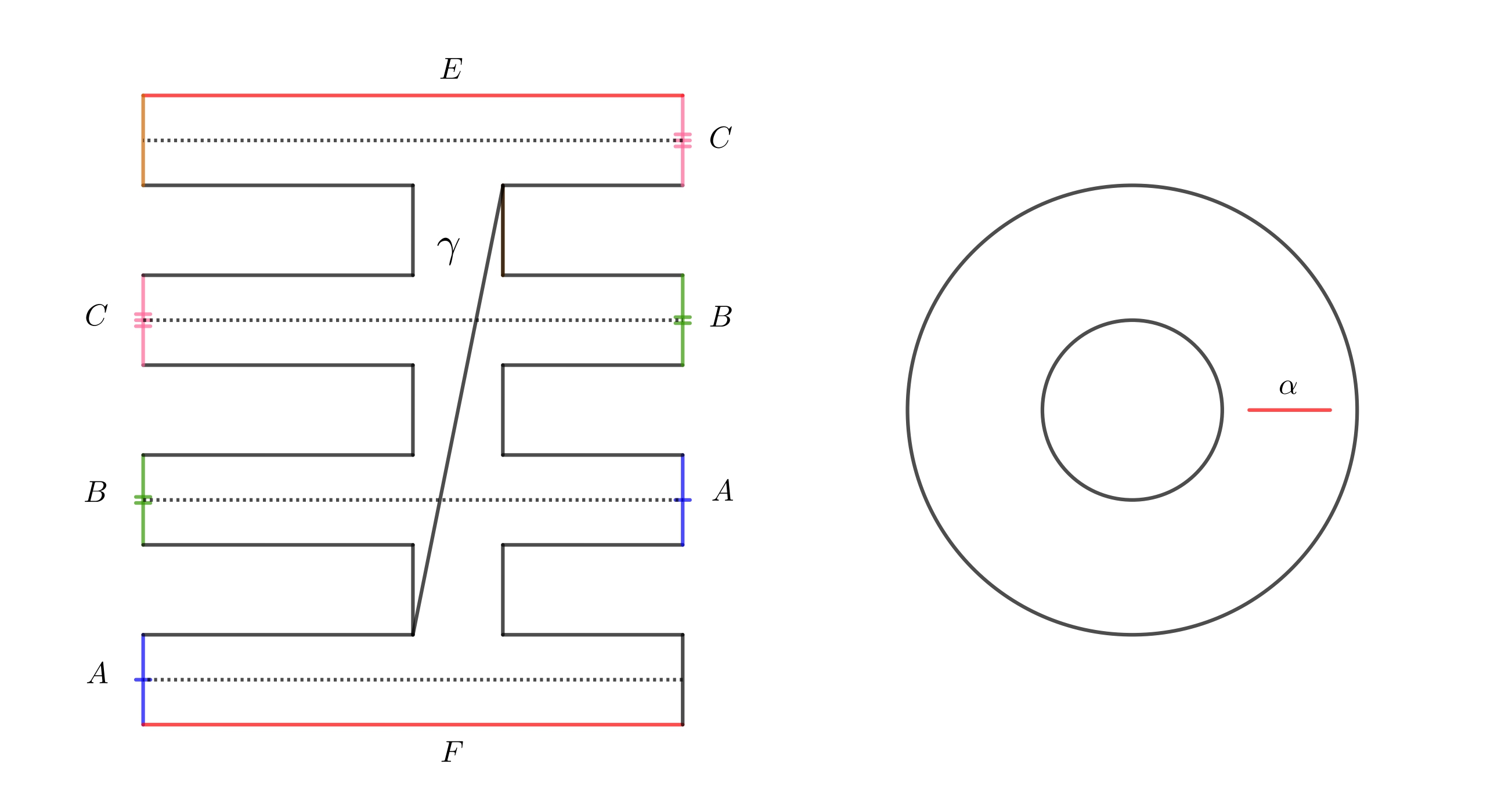}
     \caption{$\gamma$ is an example of a $2$-horizon saddle connection. Sides $A$, $B$, and $C$ are glued by translations; the other sides are glued to Hopf tori. In this construction sides $E$ and $F$ are glued to a radial segment $\alpha$ contained in a Hopf torus in order to trap the directional flows. The dotted line is an example of a directional flow that crosses $\gamma$ exactly twice.}
     \label{Figure 2-horizon}
 \end{figure}

    We begin with the construction of $M$. It will be built from a Euclidean geodesic polygon $P_r$, on which we will perform the slit construction in order to create traps for the directional flows.\\
    
    Take $r\in (0,1)$ and consider $P_r\subset \mathbb
    R^2$ to be the Euclidean geodesic polygon determined by $$\left( \bigcup_ {i=0}^{k+1} [0,r]\times [2i, 2i+1]\right)\cup [r/2, 3r/4 ]\times [1, 2(k+1)]$$

    Given $l\in \{1,\dotsb, k\}$ let $\gamma_l$ be denote  the straight segment in $P_r$ joining $(r/2,1 )$ and $(3r/4, 2l+1)$. The first identifications that we will consider in $P_r$ have the goal of creating a horizontal leaf that  intersects $\gamma_l$ exactly $l$-times. To do this, for each $i\in\{0,1, \dotsb, k\}$ identify the segments $\{0\}\times [2i, 2i+1]$ and $\{r\}\times [2(i+1), 2(i+1)+1]$ by a translation.\\

    Now, in order to create traps for the directional foliations  we will perform the slit construction gluing the remaining sides of $P_r$ to radial segments contained in Hopf tori. This gluing is performed as follows:

    Consider $H^1, \dotsb, H^{2(k+1)+1}, V^1, \dotsb, V^{k+2}$ Hopf tori with $\alpha_i\subset H ^i$ and $\beta_j\subset V^j$ horizontal and vertical segments respectively, and make the following surgeries and identifications:

    \begin{itemize}

         \item For each $i\in \{1, \dotsb k+1\}$ glue $[0, r/2]\times \{i\}$ and $[0, r/2]\times \{2(k+1)+1-i\}$ with $\alpha_{i}\subset H^i$.

         \item For each $i\in \{1, \dotsb, k+1\}$ glue $[3r/4, r]\times \{i\}$ and $[3r/4, r]\times \{2(k+1)+1-i\}$ with $\alpha_{k+1+i}\subset H^{k+1+i}$.

         \item Glue $[0, r]\times \{0\}$ and $[0,r]\times \{2(k+1)+1\}$ with $\alpha_{2(k+1)+1}\subset H^{2(k+1)+1}$.

         \item For each $i\in \{1, \dotsb k\}$ glue $\{3r/4\}\times [2i-1, 2i]$ and $\{r/2\}\times [2i+1, 2(i+1)]$ with $\beta_i\subset V^i$.
         
         \item Glue $\{r\}\times [0,1]$ and $\{r/2\}\times [1,2]$ with $\beta_{k+1}\subset V^{k+1}$.
         
         \item Glue $\{0\}\times [2(k+1), 2(k+1)+1]$ and ${3r/4}\times [2(k+1)-1, 2(k+1)]$ with $\beta_{k+2}\subset V^{k+2}$.

    \end{itemize}

    Note that after performing these surgeries, all the vertices of  $P_r$ are identified to a unique point in the quotient space. Thus, if we denote by $M$ the resulting dilation surface, then $M$ has a unique conic singularity whose angle is given by the sum of the inner angles of $P_r$ and $4\pi$ times the number of Hopf tori used in the construction. Therefore, $M$ has a unique conic singularity and its conical angle is $2\pi(10k+15)$.\\

    By construction, given $i, j\in \{0, \dotsb, k\}$ with $i< j$, we cannot join a point in $\{0\}\times [2i, 2i+1]$ to a point in $\{r\}\times [2j, 2j+1]$ via a straight line segment contained in $P_r$. Consequently, any trajectory along a directional foliation must descend from its starting point until it eventually hits a $1$-horizon saddle connection, preventing it from returning. This dynamical behavior, together with the gluing described above, turns $\gamma_l$ into an $l$-horizon saddle connection.
\end{proof}

Note that by rotating $P_r$ and  suitably adapting the gluing in the construction of $M$, we can also produces horizon saddle connections in any prescribed direction. This observation will be used in the proof of Theorem \ref{Main Theorem 1} and Theorem \ref{Main Theorem 2}.\\

\subsection{The Veech group}

Given a dilation surface $M$ and $f\in\Aff(M)$, as seen in Definition \ref{Definition of affine homeo}, in local coordinates, $f$ has the form $z\mapsto r A_fz +v$ for $r\in \mathbb{R}^+, v\in\mathbb{C}$ and $A_f\in \SL(2,\mathbb{R})$. Due to the the structural group in a dilation surface, the matrix $A_f$  does not depend on the choice of the local charts. This allows us to define a homomorphism

\begin{align*}
    V: \Aff(M)&\to \SL(2,\mathbb{R})\\
     f&\mapsto A_f
\end{align*}

\begin{definition}
    Let $M$ be a dilation surface, the \textbf{Veech group} of $M$, $\Gamma(M)$, is the subgroup of $\SL(2,\mathbb{R})$ defined by $\Gamma(M):=V(\Aff(M))$.
\end{definition}

The following result was shown by Duryev, Fougeron, and Ghazouani in \cite{DFG19} and highlights the contrast with the Veech group of a closed translation surface, which is always discrete:

\begin{theorem}[{\cite[Theorem~1]{DFG19}}]
    Let $M$ be a closed dilation surface of genus $g \geq 2$. Then there are two possibilities for its Veech group:
    \begin{enumerate}
        \item $\Gamma(M)$ is the subgroup of $\SL(2, \mathbb{R})$ consisting of upper triangular matrices, and $M$ is a Hopf surface.
        \item $\Gamma(M)$ is discrete.
    \end{enumerate}
\end{theorem}

An affine homeomorphism maps horizon saddle connection to horizon saddle connection. Indeed, affine homeomorphisms send singularities to singularities and saddle connections to saddle connections, and preserve the intersection number. Together with the following  result of Tahar in \cite{Tahar21}, this fact highlights the role played by horizon saddle connections in determining the algebraic structure of the Veech group of a closed dilation surface.

\begin{theorem}{\cite[Theorem~1.3]{Tahar21}}
Let $M$ be a closed dilation surface. If $M$ has saddle connections in at least three directions, then the Veech group of $M$ is finite.
\end{theorem}

\section{Proof of main results.}

This section presents the proof of the main results in this paper. The first one proves that any infinite-type topological surface can be endowed with a dilation structure with special properties: all singularities are of the same conic angle and trivial holonomy, except for three particular singularities. Moreover the surface has enough space to allow constructions that will be important in the prove of the subsequent  theorems.

The process used to construct this dilation surfaces is based on its version for translation surfaces introduced by Camilo Maluendas and Anja Randecker. See {\cite[Theorem~2]{Ran16}}.

\begin{figure}
    \centering
    \includegraphics[width=\linewidth]{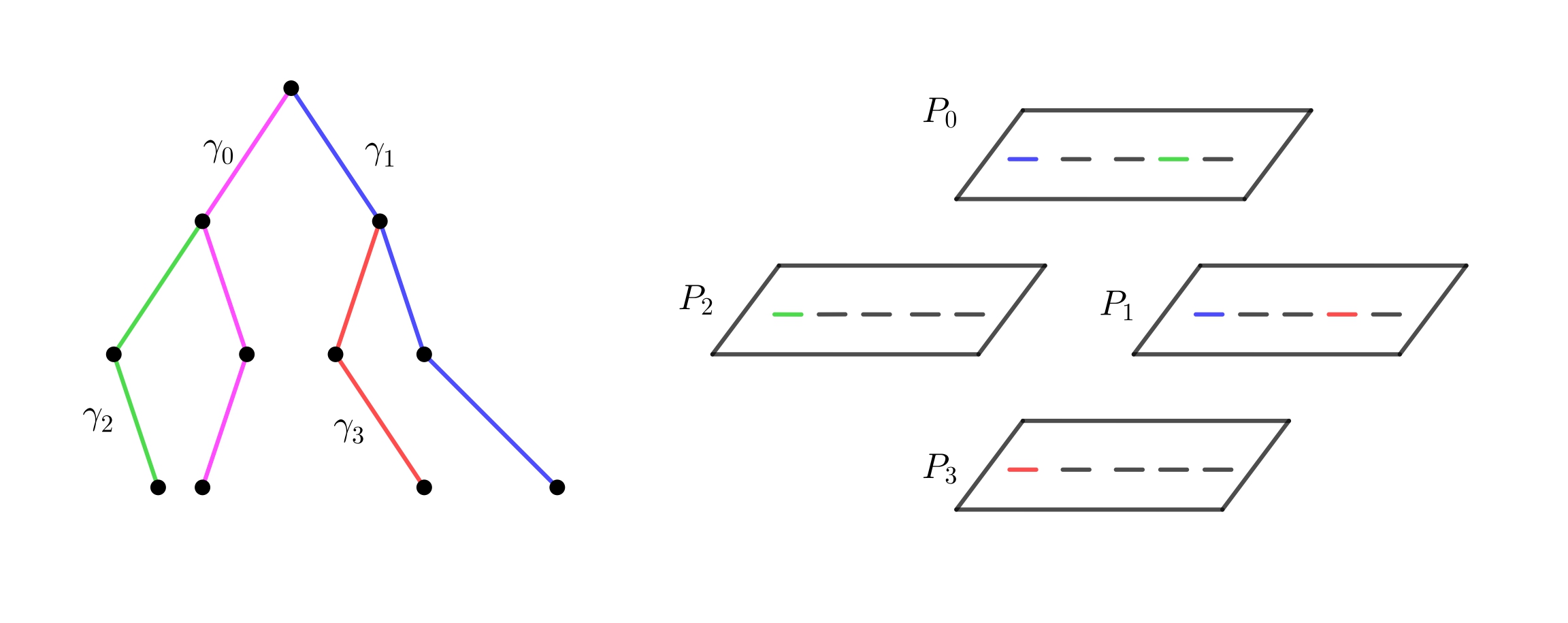}
    \caption{An example of  some gluing considering in the end-grafting construction.}
    \label{Figure_End_Grafiting_Construction}
\end{figure}

\begin{lemma}\label{Lemma_EG_Dilation}
    Let $E^\prime\subset E$ be non-empty closed subsets of the Cantor space. There exists a uncountable family of non-isomorphic dilation surfaces $\{M_r\}_{r\in(0,1)}$  such that for all $r\in (0,1)$ the following properties hold:

    \begin{enumerate}
        \item $E^g(M_r)\subset E(M_r)$ is homeomorphic to $E^\prime \subset E$.
        \item All the singularities of $M_r$ have conic angle $4\pi$ and trivial holonomy, except for three singularities $\sigma^1,\sigma^2$ and $\sigma^3$ that have  conic angle $2\pi$ and holonomy $r$, $\frac{1+r}{2r}$ and $\frac{2}{1+r}$ respectively.

        \item Given a horizontal closed half-plane $P\subset \mathbb{C}$, there exists a  embedding $I:P\hookrightarrow M_r$ that is in local coordinates a dilation and such that $\overline{I(P)}$ does not intersect any saddle connection on $M_r$.
    \end{enumerate}
\end{lemma}

\begin{proof}

Given $r\in(0,1)$ we will begin by constructing the surface $M_r$ and later, we will show that the family $\{M_r\}_{r\in (0,1)}$ satisfies the desired properties.\\

\noindent\textbf{Construction of $M_r$}. The construction of $M_r$ proceeds in three steps. Starting from a closed subset $E$ of the Cantor space, we associate a rooted subtree $T_E$ of the Cantor Tree whose space of ends is homeomorphic to $E$. On this tree, we define a family of rays $\Gamma_E=\{\gamma_n\}_{n\in\mathbb{N}_0}$ whose union covers $T_E$. 

    In the second step, we consider a collection of copies of the Euclidean plane $\{P_n\}_{n\in\mathbb{N}_0}$, on each of which we define families of slits. These are glued in a convenient manner to obtain a dilation surface whose space of ends is homeomorphic to $E^\prime\subset E$. See Figure \ref{Figure_End_Grafiting_Construction} for an example of the setup.

    Finally, a local surgery produces the conic singularities with non-trivial holonomy depending on the parameter $r$, which allows us to distinguish the resulting surfaces for different values of $r$.\\

    \underline{Step 1}. Given $E$, a closed subset of the Cantor set, there exists $T_E$, a rooted subtree of the Cantor Tree whose space of ends is $E$. Assume that $T_E$ is endowed with the path metric in which each edge has length $1$, and let $v_0$ denote the root of $T_E$.

    The construction of the family $\Gamma_E=\{\gamma_n\}_{n\in\mathbb{N}_0}$ proceeds inductively. We beging by considering a geodesic ray $\gamma_0: [0,\infty)\to T_E$ such that $\gamma_0(0)=v_0$. Now assume inductively that for $n\geq 1$ we have already constructed geodesic rays $\gamma_0,\dotsc,\gamma_{n-1}$ in $T_E$. We now construct a geodesic ray $\gamma_n$ in $T_E$.

    Let $A_n:=T_E\setminus \bigcup_{i=0}^{n-1}\gamma_i([0,\infty))$, in $A_n$ there are finitely many edges that are closest to $v_0$. Let $a_n$ be one of these edges and let $v_n^\prime$ be the vertex of $a_n$ that is closer to $v_0$. There exist $i\in \{0,\dotsc, n-1\}$ and  $k\in\mathbb{N}_0$ such that $\{\gamma_i(k), v_n^\prime\}$ is an edge of $T_E$. Now, set $v_n:=\gamma_i(k)$ and consider a geodesic ray   $\gamma_n: [0,\infty)\to A_n\cup \{v_n\}$ such that  $\gamma_n(0)=v_n$.

    By iterating this procedure, we obtain a family of geodesic rays $\Gamma_E:=\{\gamma_n\}_{n\in\mathbb{N}_0}$ satisfying
    $$\bigcup_{\gamma\in\Gamma_E}\gamma([0,\infty))=T_E.$$

    Note that, since the family $\Gamma_E$ consists of geodesic rays, for each $n\in\mathbb{N}_0$ the set of vertices of $\gamma_n$ is $\{\gamma_n(k)\}_{k\in\mathbb{N}_0}$.\\

    \underline{Step 2}. For each $\gamma_n\in \Gamma_E$ let $P_n$ be a copy of the euclidean plane $\mathbb{R}^2$ with coordinate system $(x,y)$. In the following,  we will define families of slits that we will glue in order to obtain the desired surface $M_r$.

     For each $n\in\mathbb{N}_0$, let us consider in $P_n$, three families of slits $\{f(n,k)\}_{k\in\mathbb{N}_0}$, $\{g(n,k)\}_{k\in\mathbb{N}_0}$ and $\{g^\prime(n,k)\}_{k\in\mathbb{N}_0}$ defined as follows:

    \begin{itemize} 
        \item $f(n,k)$ is the slit with end points $(6(k+1),0)$ and $(6(k+1)+\delta(n),0)$.
        \item $g(n,k)$ is the slit with end points $(6(k+1)+2,0)$ and $(6(k+1)+2+\delta(n),0)$.
        \item $g^\prime(n,k)$ is the slit with end points $(6(k+1)+4,0)$ and $(6(k+1)+4+\delta(n),0)$.
    \end{itemize}

    Where $\delta(n):=r$ if $n=0$ and $\delta(n):=1$ otherwise.\\

    Note that because of the construction of $\Gamma_E$, for each $\gamma_n\in \Gamma_E$ there are uniques $\gamma_{n^\prime}\in\Gamma_E$ and $k\in\mathbb{N}_0$ such that $\gamma_n(0)=\gamma_{n^\prime}(k)$. Then, in $\bigsqcup_{n\in\mathbb{N}_0}P_n$ glue the slit $f(n,0)$ in $P_n$ with the slit $f(n^\prime, k)$  in $P_{n^\prime}$. By carrying out this process and forgetting the unused slits, we obtain a genus zero dilation surface $M^{\prime\prime}_r$.

    If $E^\prime =\emptyset$, take $M_r^\prime:=M^{\prime\prime}_r$ and end the process. Otherwise, it will be necessary consider some extra construction in order to obtain $M_r^\prime$ with accumulated by genus space of ends homeomorphic to $E^\prime$.

    Given $\gamma_n\in \Gamma_E$ consider the linear transformation $T_n:\gamma_n \to [6,\infty)\times \{0\}\subset P_n$ that maps the edge $\{\gamma_n(0),\gamma_n(1)\}$ to the segment $[6,12]\times\{0\}$ in $P_n$. We can glue the family $\{T_n\}_{n\in\mathbb{N}_0}$ to obtain an embedding $i:T_E\hookrightarrow M^{\prime\prime}_r$. Each $e\in E^\prime$  can be represented by a geodesic ray $\gamma_e:[0,\infty)\to T_E$ such that $\gamma_e(0)=v_0$. Now, glue the slits $g(n,k)$ and $g^\prime(n,k)$ in $P_n$ whenever these are both contained in $i(\gamma_e([0,\infty)))$. Denote by $M_r^\prime$ the dilation surface obtained after this gluing.\\

\underline{Step 3}.  Finally, we will perform a surgery in order to create the singularities with non-trivial holonomy. In $P_0\subset M^{\prime}_r$, Consider the straight segment $s_r$ with endpoints  $(0,0)$ and $(2,0)$. Cutting along $s_r$ we obtain two slits $s_r^1$ ans $s_r^2$. Now glue the straight subsegment of $s_r^1$ with endpoints $(0,0)$ and $(\frac{2r}{1+r},0)$ to the subsegment of $s_r^2$ with endpoints $(0,0)$ and $(1,0)$. Additionally, glue the subsegment of $s_r^1$ with endpoints $(\frac{2r}{1+r},0)$ and $(2,0)$ to the subsegment of $s_r^2$ with endpoints $(1,0)$ and $(2,0)$. Call $M_r$ the resulting dilation surface.

This surgery gives rise three singularities $\sigma^1, \sigma^2$ and $\sigma ^3$ of angle $2\pi$ corresponding to the points $(\frac{2r}{1+r},0), (0,0)$ and$(2,0) $ and respectively. Their holonomies are $r, \frac{1+r}{2r}$ and $\frac{2}{1+r}$ respectively.\\

Because of the construction, all the singularities on $M_r$ are conic singularities of angle $4\pi$ except for $\sigma^1, \sigma^2$ and $\sigma^3$. Moreover, given a closed horizontal half-plane $P\subset \mathbb{C}$, it is possible to embed  $P$ into $P_0$ mapping $P$ into the half-plane with coordinates $\mathbb{R}\times [1,\infty)$ this map is in local coordinates a dilation and the closure of this half-plane in $P_n$ does not intersect any saddle connection on $M_r$.\\

\noindent\textbf{$E^g(M_r)\subset E(M_r)$ is homeomorphic to $E^\prime \subset E$}. Let us prove that $M$ has the desired space of ends. For this, we will build a homeomorphism $\Psi: E\to E(M_r)$ such that $\Psi(E^\prime)=E^g(M_r)$. The construction of $\Psi$ is based on constructing an embedding $i: T_E\to M_r$ that maps proper rays in $T_E$ into proper rays in $M_r$ that are wholly contained in a horizontal line, respecting the gluing used in the construction of $M_r$. \\

\underline{Definition of $\Psi$}. In a similar way to how we did with $M^\prime_r$ we can build an embedding $i:T_E\hookrightarrow M_r$ such that for each $n\in\mathbb{N}_0$, $i|_{\gamma_n}: \gamma_n\to [6, \infty)\times\{0\}\subset P_n$  is the linear transformation that maps   the edge $\{\gamma_n(0),\gamma_n(1)\}$ to the segment $[6,12]\times\{0\}\subset P_n$.

Now, for each $e\in E$, we can choose as a representative of this end, a geodesic ray $\gamma_e: [0,\infty)\to T_E$ such that $\gamma_e(0)=v_0$. It follows that $i\circ\gamma_e$ is a proper ray in $M_r$ and then, we can consider $\Psi: E\to E(M_r)$ defined by $\Psi(e):=\e(i\circ  \gamma_e)$, where $\e(i\circ  \gamma_e)$ denotes the element in $E(S)$ represented by $i\circ\gamma_e$.\\

\underline{$\Psi$ is injective}. Given $e, e^\prime\in E$, there exist proper rays $\gamma_e, \gamma_{e^\prime}:[0, \infty)\to T_E$ with $\gamma_e(0)=\gamma_{e^\prime}(0)=v_0$ representing $e$ and $e^\prime$, respectively. If $e\neq  e^\prime$, then there exists $k\in\mathbb{N}$ such that $$ \gamma_e([k, \infty))\cap \gamma_{e^\prime}([k, \infty))=\emptyset. $$

Due to the construction of $\Gamma_E$ and $M_r$, there exists $N=N(k)\in\mathbb{N}$ such that for each $n\geq N$, if $(i\circ \gamma_e)([k, \infty))$ passes through a plane $P_n$ in $M_r$, then $(i\circ \gamma_{e^\prime})([k, \infty))$ does not pass through $P_n$, and vice versa. Set
$$
K:=(i\circ \gamma_e)([0, k])\cup (i\circ \gamma_{e^\prime})([0, k]).
$$
Then $K$ is compact, and $(i\circ \gamma_e)([k+1, \infty))$ and $(i\circ \gamma_{e^\prime})([k+1, \infty))$ lie in  different connected components of $M_r\setminus K$. Therefore, it follows that $\e(i\circ \gamma_e)\neq  \e(i\circ \gamma_{e^\prime})$.\\

\underline{$\Psi$ is surjective}. Given a proper ray $r:[0, \infty)\to M_r$, we will prove that there exists a proper ray $\gamma:[0, \infty)\to T_E$ such that $\Psi(\e(\gamma))=\e(r)$. We prove the existence of $\gamma$ by distinguishing two cases according to the behaviour of $r$.

\begin{itemize} 
    \item \underline{Case 1}. There exists $n\in\mathbb{N}_0$ such that $r$ is eventually contained in $P_n$. In this case, there exists $k\in\mathbb{N}_0$ such that $r([k, \infty))\subset P_n$, and therefore, if we take $\gamma:=\gamma_n\in\Gamma_E$, then $\Psi(\e(\gamma))=\e(r)$.

    \item \underline{Case 2}. $r$ travels through an infinite sequence of planes $\{P_{n_k}\}_{k\in\mathbb{N}_0}$, crossing from one to another in finite time. In this case, for each $k\in\mathbb{N}_0$, there exists $l_k\in\mathbb{N}_0$ such that $r$ crosses from $P_{n_k}$ to $P_{n_{k+1}}$ through the slits $f(n_k, l_k)\subset P_{n_k}$ and $f(n_{k+1}, 0)\subset P_{n_{k+1}}$. The sequence of slits $\{f(n_{k+1}, 0)\}_{k\in\mathbb{N}_0}$ determines a sequence of vertices of $T_E$, $\{\gamma_{n_{k+1}}(0)\}_{k\in\mathbb{N}_0}$, and this sequence of vertices determines an injective proper ray $\gamma:[0, \infty)\to T_E$ such that for each $k\in\mathbb{N}_0$, $\gamma(k)=\gamma_{n_{k+1}}(0)$. Then, it follows that $\Psi(\e(\gamma))=\e(r)$. 
\end{itemize}

\underline{$\Psi$ is continuous}. Let $X$ be a Hausdorff, locally compact, path-connected topological space and $K \subset X$ a compact set. Given $A$, a connected component of $X \setminus K$, let $A^*$ denote the set of ends represented by proper rays that are eventually contained in $A$.

We will show that for any compact $K \subset M_r$ and any connected component $A$ of $M_r \setminus K$, there exists a compact set $K_0$ consisting of a finite set of vertices of $T_E$, and $A_0$, a connected component of $T_E \setminus K_0$, such that $\Psi(A_0^*) = A^*$.

Because of the construction of $M_r$, there exists a subset $I\subset\mathbb{N}_0$ such that $A=(\bigsqcup_{n\in I}P_{n})\setminus K$. Moreover, since $K$ is compact, the set of indices $\{n\in I: f(n,0)\cap K\neq \emptyset \}$ is finite. Let $\{n_i\}_{i=1}^k$ denote this set of indices. Let $A_0$ be the connected component of $T_E\setminus \{\gamma_{n_i}(0)\}_{i=1}^k$ such that the rays $\{\gamma_{n}\}_{n\in I}\subset \Gamma_E$ are eventually contained in $A_0$. By definition, $A_0$ satisfies $\Psi(A_0^*)=A^*$. Thus, $\Psi:E\to E(M_r)$ is continuous.\\

Finally, since $\Psi$ is a continuous bijective map between compact Hausdorff spaces, it is a homeomorphism, and it remains to prove that $\Psi(E^\prime)=E^g(M_r)$. To see this, note that the elements of $E^\prime$ correspond to ends of $T_E$ that are represented by a proper ray $\gamma_e: [0, \infty)\to T_E$ such that $i\circ \gamma_e$ passes through an infinite number of slits $g(n,k)$ and $g^\prime(n,k)$. Thus, $\Psi(e)=\e(i \circ \gamma_e)$ is an element of $E^g(M_r)$. Therefore, $E^\prime \subset E$ is homeomorphic to $E^g(M_r)\subset E(M_r)$.\\

\noindent\textbf{Non-isomorphism of the surfaces $M_r$ and $M_{r^\prime}$}. Recall that the only singularities of $M_r$ with angle $2\pi$ have holonomy $r,\frac{1+r}{2r}$ and $ \frac{2}{1+r}$. In particular, the unordered set of holonomies of these singularities is an invariant of the dilation structure of $M_r$. However, if $r\neq r^\prime $, then $r\notin \{r^\prime, \frac{1+r^\prime}{2r^\prime}, \frac{2}{1+r^\prime}\}$, and therefore $M_r$ is not isomorphic to $M_{r^\prime}$.

\end{proof}

\begin{proof}[Proof of Theorem~\ref{Main Theorem 1}]

The proof relies on an explicit geometric construction and proceeds in four main stages. First, we construct a base dilation surface $M_{Id}$ that realizes the prescribed horizon saddle connections and carries a distinguished set of singularities and families of slits. Second, we define an affine action of $G$ on $M_{Id}$ and use it to generate a family of dilation surfaces $\{M_g\}_{g\in G}$. Third, we glue all the surfaces in the family $\{M_g\}_{g\in G}$ to obtain a dilation surface $M_r$ whose space of ends admits the same star decomposition as $S$. Finally, we describe the elements of $\Aff(M_r)$ and show that $\Aff(M_r)\cong G$ and $\Gamma(M_r)=G$.\\

\noindent\textbf{Construction of $M_{Id}$}. The construction of $M_{Id}$ proceeds in four main steps. First, we construct a dilation surface with a suitable space of ends. Second, using the slit construction, we create the desired horizon saddle connections. Third, we perform a local surgery to produce a particular singularity $\delta_{Id}$, which will later allow us to control the affine automorphisms of $M_r$. Finally, we define distinguished families of slits that we will later glue appropriately to obtain $M_r$.\\

   \underline{Step 1}. Because $E(S)$ is self-similar, there exists a star point. We begin by showing that this star point can always be chosen to lie in $E^g(S)$. Since $S$ has infinite genus, there exists $y\in E^g(S)$. Moreover, if $x\in E(S)$ satisfies $y\preccurlyeq x$, then $x\in E^g(S)$. Thus, if we take $x_{\infty}\in E(S)$ a maximal element such that $y\preccurlyeq x_{\infty}$, then it follows from Theorem \ref{Radial_Symmetry_Self_Similar_Equivalence} that $x_{\infty}\in E^g(S)$ is a star point.

    Let $\{E_n\}_{n\in\mathbb{N}}$ be a star decomposition of $E(S)$ with respect to $x_{\infty}$. Taking  $E := \overline{E_1}$ and $E^\prime = \overline{E_1} \cap E^g(S)$, by Lemma \ref{Lemma_EG_Dilation}, there exists a dilation surface $M^{\prime\prime}_{Id}$ that satisfies the following properties:
    \begin{itemize} 
        \item $E^g(M^{\prime\prime}_{Id})\subset E(M^{\prime\prime}_{Id})$ is homeomorphic to $E^\prime\subset E$.
        
        \item All the singularities of $M^{\prime\prime}_{Id}$ are conic singularities of angle $4\pi$, except for three distinguished singularities $\sigma^1_{Id}, \sigma^2_{Id}$ and $\sigma^3_{Id}$ with angle $2 \pi $ and holonomy $r, \frac{1+r}{2r}$ and $\frac{2}{1+r}$ respectively.
        
        \item   There is an upper closed semi-plane $H$ in $M^\prime_{Id}$ with coordinates $(x,y) \in \mathbb{R} \times [0,\infty)$ whose closure in $M_{Id}^{\prime\prime}$ does not intersect any of the saddle connections on $M^{\prime\prime}_{Id}$.
        
    \end{itemize}

    Let us denote by $x_{\infty, Id}$ the element of $E(M^{\prime\prime}_{Id})$ corresponding to $x_\infty$.\\

    \underline{Step 2}. We now construct the horizon saddle connections. Take  $\{d_j\}_{j\in J}$, with $J\subset\mathbb{N}$, as a enumeration of the set of directions $D$. Given $d_j\in D$, let $M^j$ be the dilation surface obtained in Lemma \ref{Construction of k horizon saddle connections} rotated so that it has  a $k_{d_j}$-horizon saddle connection in direction $d_j$. Remember that $M^j$ was built from a polygon and it has a side that is glued to a Hopf torus $H^1$ through a slit $\alpha_1\subset H^1 $ contained in a radial segment. Let $h_j$ be the direction of $\alpha_1\subset H^1 \subset M^j$ after the corresponding rotation. To glue $M^j$ to $M_{Id}^{\prime\prime}$ take $R_j\subset H^1$ a slit contained in a radial segment of direction $h_j$ diametrically opposite to $\alpha_1$. In $\{(x,y)\in H: x<0, y<4\}\subset M_{Id}^{\prime\prime} $, consider the family of slits $\{r_j\}_{j\in J}$, where $r_j$ is the slit in $H$ of length $1/2$ and direction $h_j$, with one of its endpoints at $(-4j+2,2)$.  For each $j\in J$  glue the slits $r_j\subset H$ and $R_j\subset M^j$ to obtain a dilation surface $M^{\prime}_{Id}$ such that for every $d_j\in D$, there exists a $k_{d_j}$-horizon saddle connection in direction $d_j$.  See Figure \ref{Figure_Setup_H} for an example of the setup in $H$.\\

    \underline{Step 3}. To control the affine automorphisms of the surface $M_r$ to be constructed later, we need to introduce a special singularity $\delta_{Id}$. For this purpose, we consider a Euclidean regular octagon $P$ in $\{(x,y)\in H: x<0, y>5\}\subset M^{\prime}_{Id}$. We can remove the interior of $P$ and glue its parallel sides thought  dilations composed with translations, then we obtain a new dilation surface $M_{Id}$ with a unique singularity of angle $10\pi$ (the one arising from the vertices of $P$). We denote this special singularity by $\delta_{Id}$.\\

    \underline{Step 4}. Now, we  define the families of slits that later are used to give rise $M_r$. Let $A = \{g_i\}_{i\in I}$, with $I \subset \mathbb{N}$, be a generating set for $G$.  For each $g_i \in A$, let us consider two families of slits $\{s(g_i, n)\}_{n\in\mathbb{N}}$ and $\{t(g_i, n)\}_{n\in\mathbb{N}}$ defined in $H$ as follows:

    \begin{itemize} 
        \item $s(g_i, n)$ is the horizontal slit with endpoints $(2i, 2(i+n))$ and $(2i+1, 2(i+n))$.
        \item $t(g_i,n)$ is the slit that has an endpoints at $(2(i+n),2i)$, is  parallel to $g_i^{-1}\cdot (1,0)$ and has length $1/2$. Note that this slit is  contained in the horizontal band $(\mathbb{R}^+\times  (2i-1, 2i+1))\cap \{(x,y)\in H: x>y\}.$
    \end{itemize}

    \begin{figure}
        \centering
        \includegraphics[width=\linewidth]{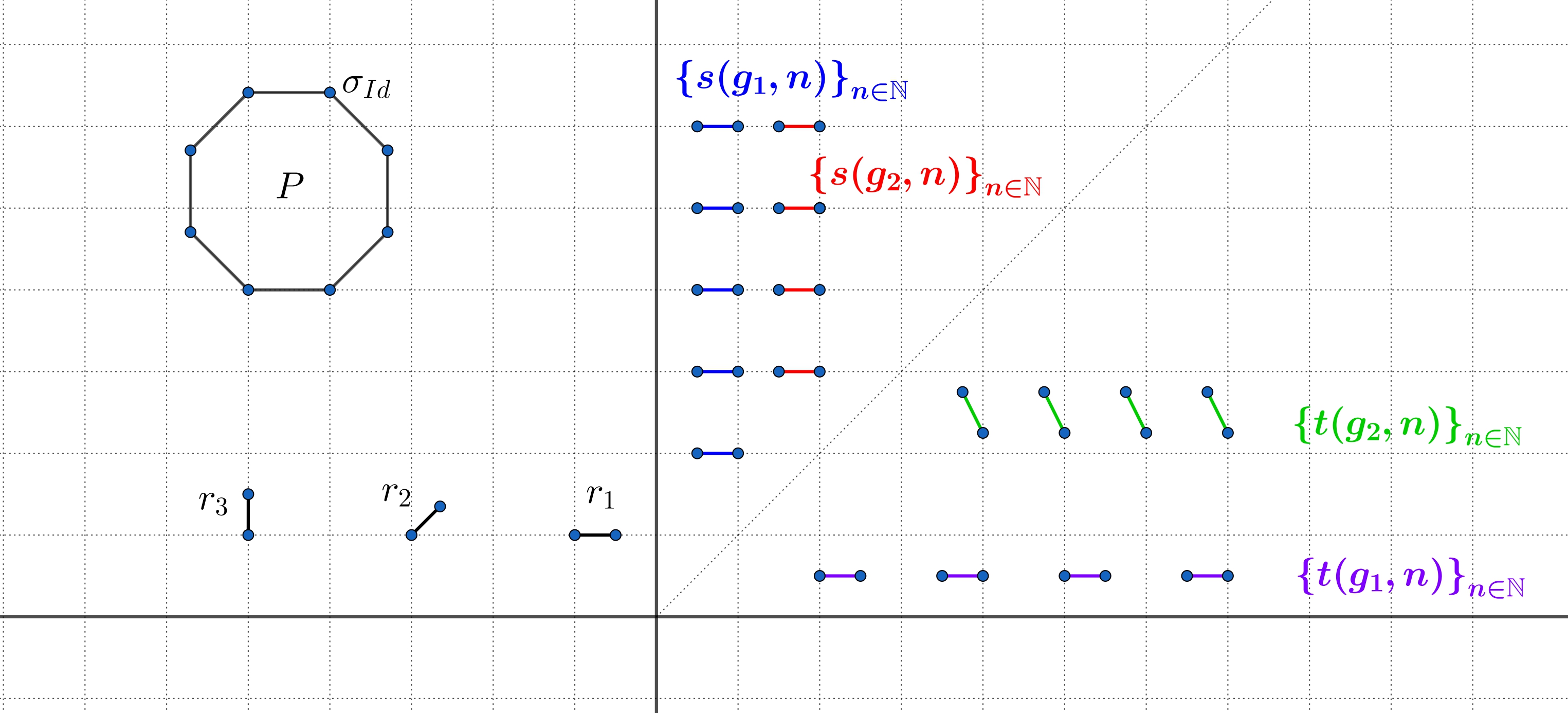}
       \caption{The geometric setup and the families of slits defined in $H$ for $G = \Bigl\langle \begin{pmatrix} 1 & 2 \\ 0 & 1 \end{pmatrix}, \begin{pmatrix} 1 & 0 \\ 2 & 1 \end{pmatrix} \Bigr\rangle$ and $D = \{0, \pi/4, \pi/2\}$.}
        \label{Figure_Setup_H}
    \end{figure}

\noindent\textbf{Construction of the family $\{M_g\}_{g\in G}$.} For each $g\in G$, we consider the  dilation surface $M_g:=g\cdot M_{Id}$, where $g$ acts on $M_{Id}$ as an affine homeomorphism whose derivative is $g$. Let $\delta_g, \sigma_g^1, \sigma_g^2, \sigma_g^3$, and $x_{\infty, g}$ denote the  images in $M_g$ of $\delta_{Id}, \sigma_{Id}^1, \sigma_{Id}^2, \sigma_{Id}^3$, and $x_{\infty, Id}$, respectively. We are already constructed a family of dilation surfaces $\{M_g\}_{g\in G}$.

Since an affine homeomorphism preserves the conic angle of a singularity,  $\delta_{g}$ is the unique singularity in $M_g$ whose conical angle is $10\pi$. Furthermore, $\sigma_g^1, \sigma_g^2$, and $\sigma_g^3$ are the only singularities in $M_g$ of conical angle $2\pi$, and by Lemma \ref{Equivalence and linear holonomy}, their holonomies lie in the set $\left\{r, \frac{1+r}{2r}, \frac{2}{1+r}\right\}$.

Note that for each $(g_i, n)\in A\times\mathbb{N}$, the slits $s(g_i, n)$ and $t(g_i, n)$ in $M_{Id}$ induce slits $g\cdot s(g_i, n)$ and $g\cdot t(g_i, n)$ in $M_g$. And thus, for each $g_i \in A$, $M_g$ is endowed with two family of slits $\{g\cdot s(g_i, n)\}_{n\in\mathbb{N}}$ and $\{g\cdot t(g_i, n) \}_{n\in\mathbb{N}}$.\\

   \noindent\textbf{Construction of $M_r$.} Because of the construction, for each $(g_i, n)\in A\times \mathbb{N}$ and $g\in G$, the slit $g\cdot s(g_i, n)\subset M_g$ is parallel to the slit $(gg_i)\cdot t(g_i,n)\subset M_{gg_i}$. Then, we can glue the surfaces $M_g$ and $M_{gg_i}$ through the families of slits $\{g\cdot s(g_i,n)\}_{n\in\mathbb{N}}$ in  $M_g$ and $\{(gg_i)\cdot(g_i,n)\}_{n\in \mathbb{N}}$ in $M_{gg_i}$. Making this process, gluing $M_g$ with $M_{gg_i}$ for all $g_i\in A$ and $g\in G$ we obtain the desired dilation surface $M_r$.\\

     Before proving that $M_r$ and $S$ have the same space of ends, we analyze how the gluings used in the construction of $M_r$ affect its space of ends.

    First, note that by performing the gluings prescribed above, the family of ends $\{x_{\infty, g}\}_{g\in G}$ gives rise to a unique end $y_\infty\in E(M_r)$. This occurs because gluing countably many copies of the Euclidean plane along two parallel, pairwise disjoint families of slits that do not accumulate results in a Loch Ness monster.

    Second, note that the process used to construct $M_r$ does not create new ends. Given $(g, g_i)\in G\times A$, let $M_g^{g_i}$ be the dilation surface obtained by gluing $M_g$ and $M_{gg_i}$ as described above. For each $e\in E(M_r)\setminus\{y_\infty\}$, there exists a proper ray with constant direction $\gamma_e:[0, \infty)\to M_r$ representing the end $e$. Due to the construction of $M_g^{g_i}$, the ray $\gamma_e$ is eventually contained entirely within either $M_g$ or $M_{gg_i}$. Consequently, any end in $M_r$ distinct from $y_\infty$ comes from a unique end in $M_g$ or $M_{gg_i}$.\\

    We now show that $M_r$ and $S$ have the same space of ends. To do so, we construct a homeomorphism $$\phi: \bigsqcup_{n\in\mathbb{N}}E_n\sqcup \{x_\infty\}\to E(M_r).$$

    For each $n\in\mathbb{N}$, we can construct an embedding $\phi_n: E_n\sqcup \{x_\infty\}\hookrightarrow E(M_r)$ with a closed image  such that $\phi_n(x_ \infty )=y_{\infty}$. We make this construction as follows:

    \begin{itemize} 
        \item Since $G$ is infinite countable, we can consider $\{g_n\}_{n\in\mathbb{N}}$ as an enumeration for $G$. By definition,  $E(M_{g_n})$ is homeomorphic to $ E_n\sqcup \{x_\infty\}$. Thus, there exists a homeomorphism $\psi_n : E_n\sqcup \{x_\infty\}\to E(M_{g_n})$ such that $\psi_n(x_\infty)=x_{\infty, g_n}$.
        
        \item By the construction of $M_r$, there exists an embedding $\varphi_n:E(M_{g_n})\hookrightarrow E(M_r)$ with a closed image  such that $\varphi_n(x_{\infty, g_n})=y_\infty$.

        \item Then, defining $\phi_n:=\varphi_n\circ \psi_n$ we obtain the embedding with the desired properties.
    \end{itemize}

    Note that for any distinct $m, n\in\mathbb{N}$, we have that $\phi_n(E_n\sqcup \{x_\infty\})\cap \phi_m(E_m\sqcup \{x_\infty\})=\{y_\infty\}$ then, we can glue the family $\{\phi_n\}_{n\in\mathbb{N}}$ to obtain an embedding $\phi: \bigsqcup_{n\in\mathbb{N}}E_n\sqcup \{x_\infty\}\to E(M_r)$ such that $\phi(x_\infty)=y_\infty$.

    Since the process used to construct $M_r$ does not create new ends, $\phi$ is surjective and maps $E^g(S)$ to $E^g(M_r)$. Since $\phi$ is a continuous bijection between compact Hausdorff spaces, we can conclude that it is a homeomorphism. Finally, because both $S$ and $M_r$ have infinite genus, it follows from Theorem \ref{Theorem classification of surfaces} that $M_r$ is homeomorphic to $S$.\\

    We  have already constructed a family of dilation surfaces $\{M_r\}_{r\in(0,1)}$. All surfaces in this family are homeomorphic to $S$, but their dilation structures are not isomorphic. To see this, note that by Lemma \ref{Equivalence and linear holonomy}, all singularities in $M_r$ with a conic angle of $2\pi$ have holonomies lying in the set $\left\{ r, \frac{1+r}{2r}, \frac{2}{1+r} \right\}$. However, for any $r^\prime\in(0,1)$ with $r\neq r^\prime$, we have that $r^\prime\notin\left\{ r, \frac{1+r}{2r}, \frac{2}{1+r} \right\}$, and therefore $M_r$ is not isomorphic to $M_{r^\prime}$.\\

    \noindent\textbf{$\Gamma(M)=G$ and $\Aff(M)\cong G$}.We begin with a description of $\Aff(M_r)$. Given $g\in G$, let $T_g \in \Aff(M_r)$ be the affine homeomorphism such that for each $h\in G$, the restriction $T_g|_{M_h}: M_h\to M_{hg}$ is, the map induced by the  affine action of  matrix $g$ in $M_h$. We will prove that $\Aff(M_r)=\{T_g\}_{g\in G}$.

     Take $T \in\Aff(M_r)$. Since the set $\{\delta_g\}_{g \in G}$ consists of the only singularities of $M_r$ with  conical angle $10\pi$, there must exist  $g \in G$ such that $T(\delta_{Id}) = \delta_g$. We will prove that $T = T_g$. Note that the only saddle connections joining $\delta_{Id}$ to itself are the saddle connections corresponding to the sides of the octagon $P$ used to create $\delta_{Id}$. Now, $T_g^{-1}(\delta_g) = \delta_{Id}$, and thus $T_g^{-1} \circ T$ fixes $\delta_{Id}$. Moreover, $T_g^{-1} \circ T$  fixes all the saddle connections that join $\delta_{Id}$ to itself;  then, in local coordinates, $T_g^{-1} \circ T$ is of the form $z\mapsto \lambda z+v$. Consequently, $T_g^{-1}\circ T$ is a biholomorphism from $M_r$ to itself  that point-wise fixes an infinite set (a side of the octagon) with an accumulation point (the singularity $\delta_{Id}$). Therefore, $T_g^{-1}\circ T=Id_{M_r}$ and  hence $T = T_g$.

    Recall that the Veech group of $M_r$ is the image of the homomorphism $V: \Aff(M_r) \to \SL(2, \mathbb{R})$ that maps $f$ to the matrix $A_f$ (see Definition \ref{Definition of affine homeo}). Due to the structure group of a dilation surface, this map is well defined. Now, as seen above, $\Aff(M_r)=\{T_g\}_{g\in G}$, and by the construction of the map $T_g$, $V(T_g)=g$. Thus, the Veech group of $M_r$ is $\Gamma(M_r)=V(\{T_g\}_{g\in G})=G$.

    Finally, note that if $g,h\in G$, then $T_{gh}=T_g\circ T_h$ and $T_{g^{-1}}=T_g^{-1}$. Consequently, the map $\varphi: G\to \Aff(M_r)$ defined by $\varphi(g):=T_g$ is an epimorphism. The injectivity of $\varphi$ follows from the fact that if $T_g=T_h$, then $V(T_g)=V(T_h)$, which implies $g=h$. Therefore, $\Aff(M_r)\cong G$.
    
    \end{proof}

Note that in the proof of Theorem \ref{Main Theorem 1}, it is important that $G$ is an infinite group, because this group allow us to glue the family $\{M_g\}_{g\in G}$ to obtain a surface whose space of ends is  $\bigsqcup_{n\in\mathbb{N}} E_n\sqcup \{x_\infty\}$. For this reason, if we want to realize the finite subgroups of $\SL(2,\mathbb{R})$ as Veech groups we  need to consider a different construction. This is described in the proof Theorem \ref{Main Theorem 2}.

\begin{proof}[Proof of Theorem~\ref{Main Theorem 2}]

    We follow a similar strategy as Theorem \ref{Main Theorem 1}, we first construct a dilation surface $M_{Id}$ with the prescribed set of horizon saddle connections, specific singularities and marked slits. We then generate the family $\{M_g\}_{g\in G}$ and glue them together to obtain the desired $M_r$.\\

    \noindent\textbf{Construction of $M_{Id}$.} The construction of $M_{Id}$ proceeds in four steps: (1) decomposing a genus-zero surface $S_E$ whose space of ends is $E$ into a one-ended subsurfaces, (2) endowing them with compatible dilation structures and gluing then back to obtain a dilation structure on $E$, (3) creating horizon saddle connections and distinguised singularities:  three singularities with non-trivial holonomy and $\delta_{Id}$, and (4) defining several families of slits required to construct $M_r$.\\

    \underline{Step 1}. Given $E:=E(S)=E^g(S)$, consider an embedding $\iota: E\hookrightarrow \mathbb{S}^2$ and define $S_E:=\mathbb{S}^2\setminus \iota(E)$. Consider $\{c_i\}_{i\in I}$ a collection of simple closed curves on $S_E$ that forms a basis for $H_1(S_E,\mathbb{Q)}$. And for each $i\in I$, let $A_i$ be a regular neighborhood of $c_i$ homeomorphic to an open annulus and such that $A_i\cap A_j=\emptyset$ whenever $i\neq  j$.

    We can choose the families $\{c_i\}_{i\in I}$ and $\{A_i\}_{i\in I}$ in such way that the complement of $\bigsqcup_{i\in I}  A_i$ on $S_E$ is the disjoint union of connected one-ended genus zero subsurface $\{S_j\}_{j\in J}$ such that $S_j$ has $\alpha_j\in \mathbb{N}\cup\{\infty\}$ boundary components and $\bigsqcup_{j\in J}E(S_j)$ is dense in $E(S_E)=E.$\\

    \underline{Step 2}.  For each $j\in J$, let $P_j^{Id}$ be a copy of the Euclidean plane $\mathbb{R}^2$ with coordinate system $(x,y)$. In  $P_j^{Id}$ consider the family of slits $\{b_j(m)\}_{1\leq m \leq \alpha_j}$, where  $b_j(m)$ is the horizontal slit with endpoints $(2m,0)$ and $(2m+1,0)$.

     Suppose that for every $j\in J$, $\partial S_j=\gamma_{j_1}\sqcup \cdots \sqcup\gamma_{j_{\alpha_j}}$.  By construction, for each $i\in I$, there exist $j, j^\prime\in J$, $1\leq m_j\leq \alpha_j$ and $1\leq m_{j^\prime} \leq \alpha_{j^\prime}$ such that $\partial \overline{A_i}=\gamma_{j_m}\sqcup \gamma_{j{\prime _{m^\prime}}}$. Then, in $\bigsqcup_{j\in J}P_j^{Id}$ glue the slit $b_j(m)\subset P_j^{Id}$ to the slit $b_{j^\prime}(m^\prime)\subset P_{j^\prime}^{Id}$.  By repeating this process, gluing $S_j$ with $S_{j^\prime}$ whenever they share a boundary associated with the same cylinder, we obtain a dilation surface $M_{Id}^{\prime}$. 
     
    From a topological perspective, we can identify each annulus $A_i$ with a small neighborhood of the glued slits $b_j(m)$ and $b_{j^\prime}(m^\prime)$. Consequently, the process described above is equivalent to gluing back the pieces $\{S_j\}_{j\in J}$ and $\{A_i\}_{i\in I}$. Therefore,  $M_{Id}^{\prime}$ is homeomorphic to $S_E$.\\

     \underline{Step 3}. Fix $j\in J$,  performing a similar construction to that of proof of Theorem~\ref{Main Theorem 1} step 2, we can generate the horizon saddle connections in $P_j^{Id}\subset M_{Id}^\prime$ in the desired set of directions. Subsequently, following the same procedure as in the proof of Lemma~\ref{Lemma_EG_Dilation} step 3, In $P_j ^{Id}\subset M_{Id}^\prime$ we create three dilation singularities of angle $2\pi$ and holonomy  $r, \frac{1+r}{2r}, \frac{2}{1+r}$, which we denote by  $\sigma_{Id}^1, \sigma_{Id}^2$, and $\sigma_{Id}^3$. 

    In $\{(x,y)\in P_j^{Id}: x<0, y>5\}$, consider a geodesic Euclidean regular octagon $P$. Removing the interior of $P$ and identifying its parallel sides, we obtain a dilation surface $M_{Id}$ with a unique conic singularity of angle $10\pi$ (obtained form $P$). We call $\sigma_{Id}$ this special singularity. \\

   \underline{Step 4}. Assume that $G=\{g_1, \dotsb, g_N\}$. For each $g_l\in G, j\in J$  consider the families of slits $\{s_j(g_l, n)\}_{n\in\mathbb{N}}$ and $\{t_j(g_l, n)\}_{n\in\mathbb{N}}$ defined as follows:

   \begin{itemize} 
       \item $s_j(g_l, n)\subset P_j^{Id}$ is the horizontal slit with endpoints $(2l, 2(l+n))$ and $(2l+1, 2(l+n))$.

       \item $t_j(g_l, n)\subset P_j^{Id}$ is that has an endpoint at $(2l+2n, 2l)$, is parallel to $g_l^{-1}\cdot(1,0)$ and has length $1/2$. Note that all these slit are whole contained in $(\mathbb{R}^+\times (2l-1, 2l+1))\cap\{(x,y)\in P_j^{Id}: x >y \}$.
   \end{itemize}

    \noindent\textbf{Construction of $M_r$.} Given  $g\in G$, define $M_g:=g\cdot M_{Id}$ where, for each $j\in J$,  $g$ acts on  $P_j^{Id}\subset M_{Id}$ as an affine homeomorphism whose derivative is $g$, we call $P_j^g$ the respective image of $P_{j}^{Id}$ through this action. Let us denote by $\delta_g, \sigma_g^1, \sigma_g^2$ and $\sigma_g^3$ the respective images in $M_g$ of $\delta_{Id}, \sigma_{Id}^1, \sigma_{Id}^2$ and $\sigma_{Id}^3$.  Following this procedure we obtain a family of dilation surfaces $\{M_g\}_{g\in G}$.

    For any $g\in G$, $\delta_g$ is the unique singularity of $M_g$ of angle $10\pi$. Moreover, $\sigma_g^1, \sigma_g^2$ and $\sigma_g^3$ are the only singularities in $M_g$ whose conic angle is $2\pi$ and by Lemma \ref{Equivalence and linear holonomy}, their holonomies lie in $\{r, \frac{1+r}{2r}, \frac{2}{1+r}\}$.

    Note that for each $(g_l, n, j)\in G\times\mathbb{N}\times J$, the slits $s_j(g_l, n)$ and $t_j(g_l, n)$ in $P_j^{Id}\subset M_{Id}$ induce slits $g\cdot s_j(g_l, n)$ and $g\cdot t_j(g_l, n)$ in $ P_j^g\subset M_g$. And thus, for each $(g_l, j) \in G\times J$, $ P_j^g\subset  M_g$ is endowed with two families of slits $\{g\cdot s_j(g_l, n)\}_{n\in\mathbb{N}}$ and $\{g\cdot t_j(g_l, n) \}_{n\in\mathbb{N}}$. 

    By definition, for each $(g_l, g_k)\in G\times G$ and every $(n, j)\in  \mathbb{N}\times J$, the slit $g_k\cdot s_j(g_l, n)\subset  M_{g_k}$ is parallel to the slit $(g_kg_l)\cdot t_j(g_l, n)\subset  M_{g_kg_l} $. Then, we can glue the surfaces $M_{g_k}$ and $M_{g_kg_l}$ through the families of slits $\{g_k\cdot s_j(g_l, n)\}_{n\in\mathbb{N}}$ in $M_{g_k}$ and $\{(g_kg_l)\cdot t_j(g_l, n)\}_{n\in\mathbb{N}}$ in $M_{g_kg_l}$. Carrying out this process, gluing $M_{g_k}$ and $M_{g_kg_l}$ for all $g_k, g_l\in G$ we obtain the desired dilation surface $M_r$.\\

    Now, we give a description of the elements in $\Aff(M_r)$. Given $g\in G$, let $T_g\in \Aff(M_r)$ be the affine homeomorphism such that for each $h\in G$, the restriction $T_g|_{M_h}: M_h\to M_{hg}$ is the map induced by the affine action of the matrix $g$ on $M_h$. Then, following the same procedure as in Theorem \ref{Main Theorem 1} we can prove that $\Aff(M_r)=\{T_g\}_{g\in G}, \Gamma(M_r)=G$ and  $\Aff(M_r)\cong G$ .\\

    Finally, we know that for a given $r\in (0,1)$, the family $\{\sigma_g^1, \sigma_g^2, \sigma_g^3\}_{g\in G}$ is formed by the only singularities in $M_r$ of conical angle $2\pi$ and their holonomies, lie in the set $\{r, \frac{1+r}{2r}, \frac{2}{1+r}\}$. However, if we take $r^\prime\in (0, 1)$ with $r^\prime\neq r$, then $r^\prime\notin \{r, \frac{1+r}{2r}, \frac{2}{1+r}\}$. Thus, by Lemma \ref{Equivalence and linear holonomy}, the family $\{M_r\}_{r\in(0,1)}$ consists of non-isomorphic dilation surfaces.\\

    \noindent\textbf{$M_r$ is homeomorphic to $S$.} The proof consists of showing that $M_r$ is a surface of infinite genus such that $E(M_r)=E^g(M_r)=E$. Thus, by Theorem \ref{Theorem classification of surfaces}, we can conclude that $M_r$ is homeomorphic to $S$.\\

Recall that $M_r$ is obtained by gluing a finite family of dilation surfaces $\{M_g\}_{g\in G}$. We begin by analyzing the ends of $M_{Id}$. $M_{Id}$ is obtained from $M_{Id}^\prime$ after adding the prescribed horizon saddle connections and additional singularities (see step 3). This process does not create new ends; it merely adds genus. Thus, we can conclude that $E(M_{Id})=E(M_{Id}^\prime)$ (possibly replacing some planar ends with ends accumulated by genus).\\

$M_{Id}^\prime$ is obtained by appropriately gluing a collection of Euclidean planes $\{P_j^{Id}\}_{j\in J}$. Topologically, $M_{Id}^\prime$ is homeomorphic to $S_E$, a genus-zero topological surface whose space of ends is exactly $E$. $S_E$ has a family of one-ended subsurfaces $\{S_j\}_{j\in J}$ such that the union of their ends, $\bigsqcup_{j\in J}E(S_j)$, is dense in $E(S_E)$. In the construction of $M_{Id}^\prime$, we associate each Euclidean plane $P_j^{Id}\subset M_{Id}^\prime$ with the subsurface $S_j\subset S_E$. Consequently, the union of the ends $\bigsqcup_{j\in J}E(P_j^{Id})$ is dense in $E=E(M_{Id}^\prime)=E(M_{Id})$.\\

Note that for any $g\in G$, $M_g$ is obtained from $M_{Id}$ through an affine homeomorphism whose derivative is $g$. Thus, we can conclude that for each $g\in G$, $E(M_g)=E(M_{Id})=E$.\\

Now, we return to the construction of $M_r$. This was built by suitably gluing, for each $j\in J$, the family $\{P_j^g\}_{g\in G}$. For each $g, h\in G$, $P_j^g\subset M_g$ and $P_j^h\subset M_h$ are glued along two infinite families of parallel slits that do not accumulate. Thus, gluing the family $\{P_j^g\}_{g\in G}$ results in a Loch Ness Monster. We denote the resulting surface by $P_j^G$.

Note that the gluing process used to obtain $M_r$ does not create new ends, and thus, since $E(M_g)=E$ for all $g\in G$, we have $E(M_r)=E$.

Recall that for each $g\in G$, the union $\bigsqcup_{j\in J}E(P_j^g)$ is dense in $E(M_g)=E$; consequently, the union $\bigsqcup_{j\in J}E(P_j^G)$ is dense in $E(M_r)=E$. The family $\{E(P_j^G)\}_{j\in J}$ consists entirely of ends accumulated by genus, and because the subset of ends accumulated by genus is a closed subset of the space of ends, we can conclude that $E^g(M_r)=E(M_r)=E$.

\end{proof}

\begin{proof}[Proof of Theorem~\ref{Main Theorem 3}]

    Let $M_r$ be the surface obtained  applying Lemma \ref{Lemma_EG_Dilation} to the pair $E^g(S)\subset E(S)$. Given $g\in G, n\in\mathbb{N}$ let $T_{g, n}$ be the linear map in the plane $P_n\subset M$ induced by the matrix $g$, this map respect the gluing using for the construction of $M_r$ in Lemma \ref{Lemma_EG_Dilation} and then we can glue the family $\{T_{g, n}\}_{n\in\mathbb{N}}$ to obtain a map $T_g\in \Aff(M_r)$ such that $V(T_g)=g$ and then $G\subset V(\Aff(M_r))=\Gamma(M_r)$. Finally, note that by the construction of $M_r$, all its saddle connection are horizontal and thus $\Gamma(M_r)  \subset  G$.

  \end{proof}

\bibliographystyle{alpha} 
\bibliography{sample.bib}

\end{document}